\documentclass[11pt, a4paper]{amsart}
\usepackage[dvips]{epsfig}
\usepackage{amsmath}
\usepackage{amssymb}
\usepackage{amsfonts}
\usepackage{amsthm}
\usepackage{amsbsy}
\usepackage{amsgen}
\usepackage{amscd}
\usepackage{amsopn}
\usepackage{amstext}
\usepackage{amsxtra}

\newtheorem{theorem}{Theorem}[section]
\newtheorem{proposition}[theorem]{Proposition}
\newtheorem{lemma}[theorem]{Lemma}
\newtheorem{corollary}[theorem]{Corollary}

\theoremstyle{definition}

\newtheorem{remark}[theorem]{Remark}

\usepackage{color}
\newcommand{\fixme}[1]{\footnote{\color{red}Fixme: #1}}

\newcommand {\Z} {\mathbb{Z}}

\newcommand {\s} {\mathcal{S}}
\newcommand {\R} {\mathbb{R}}

\newcommand {\E} {\mathbb{E}}

\newcommand {\B} {\mathcal {B}}

\newcommand {\M} {\mathcal {M}}

\newcommand {\eigspc} {\mathcal{E}_n}
\newcommand {\eigspcdim} {\mathcal{N}}
\newcommand {\eigval} {E_{n}}
\newcommand {\spheredim}{2}
\newcommand {\sphere}{\s^\spheredim}
\newcommand {\spherem}{\s^m}

\newcommand {\funcdercorr} {\frac{\eigval}{\spheredim}}

\newcommand{\length}{\mathcal{Z}}
\newcommand{\var}{\operatorname{Var}}
\newcommand{\divg}{\operatorname{div}}

\newcommand{\tr}{\operatorname{tr}}
\newcommand{\vol}{\operatorname{Vol}}
\newcommand{\len}{\operatorname{len}}

\begin{document}

\begin{abstract}
Using the multiplicities of the Laplace eigenspace on the sphere (the space of spherical harmonics) we endow the space with Gaussian
probability measure. This induces a notion of random Gaussian spherical harmonics of degree $n$ having Laplace eigenvalue $E=n(n+1)$.
We study the length distribution of the nodal lines of random spherical harmonics.

It is known that the expected length is of order $n$. It is natural
to conjecture that the variance should be of order $n$, due to the natural scaling.
Our principal result is that, due to an unexpected cancelation, the variance of the nodal length
of random spherical harmonics is of order $\log{n}$. This behaviour is consistent with the one predicted by Berry for nodal lines on chaotic billiards
(Random Wave Model). In addition we find that a similar result is applicable for ``generic" linear statistics
of the nodal lines.

\end{abstract}

\title[Fluctuations of nodal length]
{Fluctuations of the nodal length of random spherical harmonics}
\author{Igor Wigman}
\address{Centre de recherches math\'ematiques (CRM),
Universit\'e de Montr\'eal C.P. 6128, succ. centre-ville Montr\'eal,
Qu\'ebec H3C 3J7, Canada \newline
currently at \newline
Institutionen f\"{o}r Matematik, Kungliga Tekniska h\"{o}gskolan (KTH),
Lindstedtsv\"{a}gen 25, 10044 Stockholm, Sweden} \email{wigman@kth.se}
\thanks{The author is supported by a CRM ISM fellowship, Montr\'eal and
the Knut and Alice Wallenberg Foundation, grant KAW.2005.0098}


\maketitle


\section{Introduction}

Nodal patterns (first described by Ernest Chladni in 18th century) appear in
many problems in engineering, physics and natural sciences: they describe the
sets that remain stationary during vibrations. Hence, their importance in such
diverse areas as musical instruments, mechanical structures, earthquake study
and other areas. They also arise in the study of wave propagation and in astrophysics;
this is a very active and rapidly developing research area.
Let $(\M,g)$ be a compact manifold and $f:\M\rightarrow\R$ be a real
valued function. The nodal set of $f$ is its zero set $f^{-1}(0)=\{x\in
\M:\: f(x)=0 \}$.

The most important or fundamental case is that of $f$ being the eigenfunction of the Laplace-Beltrami
operator on $\M$
\begin{equation}
\label{eq:Laplace}
\Delta_{g}f+Ef= 0,
\end{equation}
with $E\ge 0$. In this case it is known ~\cite{Cheng}, that generically, the nodal sets are smooth submanifolds of $\M$ of
codimension $1$. For example, if $\M$ is a surface, the nodal sets
are smooth curves, also called the nodal {\em lines}. One is interested in studying their volume (i.e.
the length of the nodal line for the $2$-dimensional case) and
other properties for highly excited eigenstates. Yau
conjectured ~\cite{Y1,Y2} that the volume of the nodal set is commensurable to $\sqrt{E}$
in the sense that there exist constants $c_{\M},C_{\M}>0$ such that
if $f$ satisfies \eqref{eq:Laplace} then
\begin{equation}
\label{eq:Yau}
c_{\M} \sqrt{E} \le \vol(f^{-1}(0) ) \le C_{\M} \sqrt{E}.
\end{equation}
The lower bound was proved by Bruning and Gromes
~\cite{Bruning-Gromes} and Bruning ~\cite{Bruning} for the planar
case. Donnelly and Fefferman
~\cite{Donnelly-Fefferman} finally settled Yau's conjecture for real
analytic metrics. However, the general case of a smooth manifold is
still open.

\subsection{Spherical Harmonics}

In this paper, we will concentrate on the nodal sets on the sphere.
It is well known that the eigenvalues $E$ of the Laplace
equation
\begin{equation*}
\Delta f +E f = 0
\end{equation*}
on the $m$-dimensional sphere $\spherem$ are all the numbers of the form
\begin{equation}
\label{eq:eigval def} \eigval^{m} = n(n+m-1),
\end{equation}
where $n$ is an integer. Given a number $\eigval^{m}$, the corresponding
eigenspace is the space $\eigspc^{m}$ of spherical harmonics of
degree $n$. Its dimension is given by
\begin{equation}
\label{eq:eigspcdim asymp} \eigspcdim=\eigspcdim_{n}^{m}=
\frac{2n+m-1}{n+m-1} {n+m-1 \choose
m -1} \sim \frac{2}{(m-1)!}n^{m-1}.
\end{equation}
Given an integral number $n$, we fix an $L^2(\spherem)$ orthonormal
basis of $\eigspc^{m}$
\begin{equation*}
\eta_{1}^{n;m} (x), \, \eta_{2}^{n;m}
(x),\ldots ,\eta_{\eigspcdim_{n}^{m}}^{n;m} (x),
\end{equation*}
giving an identification $\eigspc^{m}\cong\R^{\eigspcdim_{n}^{m}}$. For further reading on the spherical harmonics we refer the reader
to ~\cite{AAR}, chapter 9.

\subsection{Random model}

We consider a {\em random eigenfunction}
\begin{equation}
\label{eq:rand eigfnc def} f_{n}^{m}(x)= \sqrt{\frac{|\spherem|}{\eigspcdim_{n}^{m}}}
\sum\limits_{k=1}^{\eigspcdim_{n}^{m}} a_{k}\eta^{n;m}_{k}(x),
\end{equation}
where $a_k$ are standard Gaussian $N(0,1)$ i.i.d. That is, we use the
identification $$\eigspc^{m}\cong\R^{\eigspcdim_{n}^{m}}$$ to endow the space $\eigspc^{m}$ with
Gaussian probability measure $\upsilon$ as
\begin{equation*}
d\upsilon(f^{m}_{n}) = e^{-\frac{1}{2}\|\vec{a}\|^2}\frac{da_{1} \cdot\ldots\cdot
da_{\eigspcdim_{n}^{m}}}{(2\pi)^{\eigspcdim_{n}^{m}/2}},
\end{equation*}
where $\vec{a} = (a_{i})\in\R^{\eigspcdim_{n}^{m}}$ are as in \eqref{eq:rand eigfnc def}.

Note that $\upsilon$ is invariant with respect to the
orthonormal basis for $\eigspc^{m}$. Moreover, the Gaussian {\em random field} $f^{m}_{n}$ is
isotropic in the sense that for every $x_{1},\ldots x_{l}\in\spherem$ and every orthogonal
$R\in O(m+1)$,
\begin{equation}
\label{eq:f(Rx1...Rxl)d=f(x1,...xl)}
\left(f_{n}(Rx_{1}),\ldots, f_{n}(Rx_{l})\right) \stackrel{d}{ =} \left(f_{n}(x_{1}),\ldots, f_{n}(x_{l})\right).
\end{equation}

As usual, for any random variable
$X$, we denote its expectation $\E X$. For example, with
the normalization factor in \eqref{eq:rand eigfnc def}, for every $m \ge 2$, $n$ and
{\em fixed} point $x\in\spherem$, one has
\begin{equation}
\label{eq:E(f(x)^2)=1} \E [f^{m}_{n}(x)^2] =
\frac{\spherem}{\eigspcdim_{n}^{m}}\sum\limits_{k=1}^{\eigspcdim_{n}^{m}} \eta^{n;m}_k(x) ^2 = 1,
\end{equation}
a simple corollary from the addition theorem (see
~\cite{AAR}, or \eqref{eq:u(x,y) def ult} for $m=2$).

Any characteristic $X(L)$ of the nodal set
$$L=L(f^{m}_{n})=\{x\in\sphere:\: f^{m}_{n}(x)=0 \} $$ is a random variable. The most natural
characteristic of the nodal set $L_{f^{m}_{n}}$ of $f^{m}_{n}$ is, of course, its
$(m-1)$-dimensional volume $\length(f^{m}_{n})$.
The main goal of the present paper is the study of
the distribution of the random variable $\length(f^{m}_{n})$ for a random Gaussian
$f_{n}\in\eigspc$.

\subsection{Some Conventions}

Throughout the paper, the letters $x,y$ will denote either
points on the sphere $\spherem$ or spherical variables.
For $x,\, y\in\spherem$, $d(x,y)$ will stand
for the spherical distance between $x$ and $y$.
Given a set $F\subseteq\spherem$, we denote its area by $|F|$;
$\len(\mathcal{C})$ will stand for the length of a smooth curve $\mathcal{C}\subset\spherem$.
For example,
\begin{equation*}
|\sphere| = 4\pi,
\end{equation*}
and
\begin{equation*}
\length(f^{2}_{n})=\len(\{f^{2}_{n}(x)=0\}).
\end{equation*}

In this paper we are mainly concerned with the $2$-dimensional case. Therefore,
to simplify the notation we will use $f_{n}(x):=f^{2}_{n}(x)$,
and accordingly $\eigspc:=\eigspc^{2}$, $\eigval:=\eigval^{2}$, $\eigspcdim_{n}:=\eigspcdim_{n}^{2}$,
$\eta_{k}^{n}:=\eta_{k}^{n;2}$.

In this manuscript, we will use the notations $A\ll B$ and $A=O(B)$
interchangeably. If necessary, the constant involved will
depend on the parameters written in the subscript. For example, $O_{\varphi}$
or $\ll_{\varphi}$ means that the constants involved depend on the function $\varphi$.

\subsection{Nodal length and related subjects}

It is widely believed that for {\em generic chaotic} billiards, one can model the nodal lines for
eigenfunctions of eigenvalue of order $\approx E$ with
nodal lines of isotropic, monochromatic random waves of wavenumber $\sqrt{E}$
(this is called Berry's Random Wave Model or RWM).
Berry \cite{Berry 2002} found
that the expected length (per unit area) of the nodal lines for the RWM is of size
approximately $\sqrt{E}$, and he argued that the variance should
be of order $\log E$.

\vspace{3mm}

Berard ~\cite{Berard} proved that for every $m\ge 2$,
\begin{equation}
\label{eq:Elen=*sqrt(E)}
\E \left[\length(f^{m}_{n}) \right] = c_{m}\cdot \sqrt{\eigval^{m}},
\end{equation}
where
\begin{equation*}
c_{m} = \frac{2\pi^{m/2}}{\sqrt{m}\Gamma\left( \frac{m}{2}  \right)},
\end{equation*}
(see also ~\cite{Neuheisel} and ~\cite{W1}).
Furthermore, Neuheisel ~\cite{Neuheisel} established an asymptotic upper bound for the variance of the form
\begin{equation}
\label{eq:var bnd Neuheisel}
\var(\length(f^{m}_{n})) = O\left(\frac{\eigval^{m}}{n^{\frac{(m-1)^2}{3m+1}}}\right) =
O\left(\frac{\eigval^{m}}{\eigspcdim^{\frac{m-1}{3m+1}}}\right)
\end{equation}
and in our previous work ~\cite{W1}, we improved the latter to be
\begin{equation*}
\var(\length(f^{m}_{n})) = O\left( \frac{\eigval^{m}}{\sqrt{\mathcal{N}^{m}_{n}}}  \right).
\end{equation*}
Either of the bounds implies that the variance of the length, {\em normalized} so that its expected value is $1$,
vanishes with prescribed rate,
\begin{equation*}
\var\left(\frac{\length(f_{n})}{\E[\length(f_{n})]}\right) = O\left( \frac{1}{\sqrt{\mathcal{N}_{n}}}  \right)
\end{equation*}
for the latter bound.
This means that the constants $c_{\sphere}$ and $C_{\sphere}$ guaranteed by Donnely-Fefferman
\eqref{eq:Yau} may be taken as essentially equal for ``generic" eigenfunctions $f^{m}_{n}\in\mathcal{E}^{m}_{n}$, where $n$ is large.

The volume of the nodal line of a random eigenfunction on the torus
$$\mathcal{T}^{\spheredim} = \R^{\spheredim}/\Z^{\spheredim}$$ was
studied by Rudnick and Wigman ~\cite{RW} and subsequently by Krishnapur and Wigman ~\cite{KW}.
In this case, it is not
difficult to see that the expectation is again $\E [\length
(f^{\mathcal{T}^\spheredim})] = const\cdot \sqrt{E}$. Their principal
result is that as the eigenspace dimension $\eigspcdim$ grows to
infinity, the variance is bounded by $$\var(\length
(f^{\mathcal{T}^\spheredim}))=O\bigg(\frac{E}{\eigspcdim^2}\bigg),$$
and it is likely that it is asymptotic to $\var(\length
(f^{\mathcal{T}^\spheredim}))\sim *\frac{E}{\eigspcdim^2}$
for a ``generic" sequence of eigenvalues.

For generic manifolds, one does not expect the Laplacian to have any multiplicities, so that
we cannot introduce a Gaussian ensemble on the eigenspace.
Let $E_{j}$ be the eigenvalues and $\phi_{j}$ the corresponding eigenfunctions. It is well known that
the $E_{j}$ are discrete, $E_{j}\rightarrow\infty$ and $L^{2}(\M) = span\{\phi_{j} \}$.

In this case, rather than considering random eigenfunctions, one considers random
{\em combinations} of eigenfunctions with growing energy window of either type
\begin{equation*}
f^{L}(x) = \sum\limits_{E_{j}\in [0,E]} a_{j}\phi_{j}(x)
\end{equation*}
(called the long range), or
\begin{equation*}
f^{S}(x) = \sum\limits_{\sqrt{E_{j}}\in [\sqrt{E},\sqrt{E}+1]} a_{j}\phi_{j}(x),
\end{equation*}
(called the short range), as $E\rightarrow\infty$.
Berard ~\cite{Berard} and Zelditch ~\cite{Z} found that $$\E\length(f^{L}) \sim \tilde{C}_{\M} \cdot \sqrt{E}$$ and recently Zelditch
~\cite{Z} proved that $$\E\length(f^{S}) \sim \tilde{C}_{\M} \cdot \sqrt{E},$$ notably with the same constant
$\tilde{C}_{\M}>0$ for both the long and the short ranges.

For billiards (i.e. surfaces with piecewise smooth boundary), one is interested in the number of intersections of the nodal line
with the boundary, or, equivalently, the number of open nodal components. Toth and Wigman ~\cite{TW1} studied
the number of boundary intersections for random combinations of eigenfunctions $f^{L}(x)$ and $f^{S}(x)$ on
generic billiards, defined precisely as above.
They found that the expected number of the intersections is of order $\sqrt{E}$.

In the first part of this paper, we resolve the high energy asymptotic behaviour for the variance of the nodal length
for random $2$-dimensional spherical harmonics $$f_{n}=f^{2}_{n}:\sphere\rightarrow\R.$$

\begin{theorem}\fixme{Changed the leading constant}
\label{thm:var length} One has
\begin{equation}
\label{eq:var length} \var\left(\length(f_{n})\right) =
\frac{1}{32}\log{n}+O(1),
\end{equation}
asymptotically as $n\rightarrow\infty$.
\end{theorem}

For the higher dimensional sphere $\mathcal{S}^{m}\subseteq \R^{m+1}$ with $m\ge 3$, it is possible
to prove ~\cite{W2} that
\begin{equation*}
\var(\length(f^{m}_{n})) = O\left( \frac{1}{n^{m-2} }   \right) = O\left( \frac{\eigval^{m}}{n\eigspcdim^{m}_{n} }   \right),
\end{equation*}
and it is likely that
\begin{equation*}
\var(\length(f_{n}^{m})) \sim \frac{c}{n^{m-2} }
\end{equation*}
for some constant $c>0$. We intend to address the question of precise asymptotics for the higher dimensional case in the future.

\subsection{Smooth linear statistics}
\footnotemark
\footnotetext{The author wishes to thank Steve Zelditch for his suggestion to consider
the smooth linear statistics as a measure of the ``stability" for the result obtained
for the length and ``Berry's cancelation phenomenon" (discussed in section \ref{sec:Berry Canc Ph} below)}

Rather than considering the volume of the full nodal line one may choose a nice submanifold $F\subseteq\spherem$
of the sphere and consider the nodal volume
\begin{equation*}
\length^{F}(f^{m}_{n}) := \vol(\{f^{m}_{n}=0\}\cap F).
\end{equation*}
inside $F$.
More generally, let $\varphi:\spherem\rightarrow\R$ be a function. One then defines
\begin{equation*}
\length^{\varphi}(f^{m}_{n}) := \int\limits_{{f^{m}_{n}}^{-1}(0)}\varphi(x)d\vol_{{f^{m}_{n}}^{-1}(0)}(x).
\end{equation*}
The random variable $\length^{\varphi}(f^{m}_{n})$ is called a (smooth) {\em linear statistic} of the nodal set.
A priori, this definition makes sense only for {\em continuous} test function $\varphi\in C(\spherem)$, so that
the restriction $\varphi|_{{f^{m}_{n}}^{-1}(0)}\in C({f^{m}_{n}}^{-1}(0))$ is defined. Unfortunately, the class
$C(\sphere)$ of continuous functions does not contain the characteristic functions of smooth sets.
However, it is known ~\cite{GS} that
for a smooth $(m-1)$-dimensional hypersurface $\mathcal{C}$ one can define the {\em trace}
$\tr_{\mathcal{C}}(\varphi)\in L^{1}(\mathcal{C})$ of $\varphi$ for
some wider classes of functions such as $W^{1,1}(\spherem)$, the class of integrable functions with integrable
{\em weak} derivatives, even though the values of $\varphi\in W^{1,1}(\spherem)$
are defined up to measure zero spherical sets. To define the trace, one exploits the values of $\varphi$ in a small
tubular neighbourhood of $\mathcal{C}$.

Unfortunately again, the class $W^{1,1}$ does not contains the family of characteristic functions of nice sets.
As an example, let us consider the $2$-dimensional spherical disc
$F=B(N,\frac{\pi}{4})\subsetneq\sphere$ centered at the north pole of radius $\frac{\pi}{4}$, $\mathcal{C} = \partial F$
its boundary, and $\varphi = \chi_{F}$. Then the definition of $\tr_{\mathcal{C}}(\chi_{F})$ is ambiguous since
one may define it as either $0\in L^{1}(\mathcal{C})$ or $1\in L^{1}(\mathcal{C})$.
This phenomenon (i.e. the jump in $f$ occurring precisely on $\mathcal{C}$)
is typical to the class $BV(\spherem)$ of functions of bounded variation; it is known ~\cite{GS},
that for any characteristic function $\chi_{F}$ of a submanifold $F$ with $C^{2}$ boundary,
$\chi_{F}\in BV(\spherem)$, and, in addition, $W^{1,1}(\spherem)\subsetneq BV(\spherem)$.
It turns out that, despite this subtlety, one can still extend the notion of {\em average} trace $$\varphi_{\mathcal{C}}^{\pm}=\tr^{\pm}_{\mathcal{C}}(\varphi)
\in L^{1}(\mathcal{C})$$
to the full class $\varphi\in BV(\spherem)$ (see Appendix \ref{apx:BV functions} for more details). For instance, in our previous example,
$\tr^{\pm}_{\partial F}(\chi_{F}) \equiv \frac{1}{2}$. It is then natural
to define
\begin{equation*}
\length^{\varphi}(f^{m}_{n}) := \int\limits_{{f^{m}_{n}}^{-1}(0)}\varphi_{\mathcal{C}}^{\pm}(x)d\vol_{{f^{m}_{n}}^{-1}(0)}(x).
\end{equation*}

It is easy to compute the expected value of a ``generic" linear statistic, following along the lines of
the proof of ~\cite{W1}, Proposition 1.4, starting from \eqref{eq:KacRice BV}.

\begin{lemma}
For $$\varphi \in BV(\sphere)\cap L^{\infty}(\sphere)$$ we have
\begin{equation}
\label{eq:Elen=int(phi)*sqrt(eigval)}
\E\left[ \length^{\varphi}(f_{n})  \right] =  \frac{\int\limits_{\sphere}\varphi(x) dx}{2^{3/2}} \sqrt{\eigval}.
\end{equation}
\end{lemma}

\begin{remark}
\label{rem:lenvarphi=0 varphi odd}
Note that $f_{n}$ is odd or even if $n$ is odd or even respectively, so that in particular the nodal lines
are symmetric w.r.t. the involution $x\mapsto -x$. Therefore if $\varphi$ is {\em odd} then $\length^{\varphi}(f_{n})$
vanishes identically in either case.
Moreover,
\begin{equation*}
\length^{\varphi}(f_{n}) = \length^{\varphi^{ev}}(f_{n}),
\end{equation*}
where the even part of $\varphi$ $$\varphi^{ev}(x) := \frac{\varphi(x)+\varphi(-x)}{2}$$ does not vanish identically, if and only if, $\varphi$
is {\em not odd}. Therefore, we may assume that $\varphi$ is {\em even} in the first place, and we will assume
so throughout the rest of this paper.
\end{remark}

Under the assumption of continuous differentiability we have the following result for the variance of $\length^{\varphi}$
for $2$-dimensional spherical harmonics.

\begin{theorem}\fixme{Changed the leading constant}
\label{thm:var smt lin stat}
Let $\varphi:\sphere\rightarrow\R$ be a continuously differentiable even function, which does not vanish identically.
Then as $n\rightarrow\infty$
\begin{equation}
\label{eq:var(len^phi) cont diff asymp}
\var(\length^{\varphi}(f_{n})) = c(\varphi) \cdot \log{n}+O_{\|\varphi \|_{\infty},V(\varphi)}(1),
\end{equation}
where
\begin{equation}
\label{eq:c(phi) def}
c(\varphi):=\frac{\|\varphi \|_{L^{2}(\sphere)}^2 }{128 \pi} > 0,
\end{equation}
i.e. the constant involved in the ``$O$"-notation depends only on the $L^{\infty}$ norm
$\|\varphi \|_{\infty}$ and the total variation $V(\varphi)$ of $\varphi$,
and moreover, this dependency is monotone increasing.
\end{theorem}

Unfortunately, Theorem \ref{thm:var smt lin stat} does not cover the characteristic functions of nice submanifolds.
For this, we have Theorem \ref{thm:var gen lin stat}; the main idea of its proof is approximating a function
$\varphi\in BV(\sphere)$ with $C^{\infty}$ functions $\varphi_{i}$, for which we apply Theorem
\ref{thm:var smt lin stat}. We control the error term in \eqref{eq:var(len^phi) cont diff asymp}
applied to $\varphi_{i}$ using its $L^{\infty}$ norm and variation, which is why we included
this technical statement in the formulation of Theorem \ref{thm:var smt lin stat} in the first place.

\begin{theorem}\fixme{Changed the leading constant}
\label{thm:var gen lin stat}
Let $$\varphi \in BV(\sphere)\cap L^{\infty}(\sphere)$$ be a not identically vanishing even function.
Then as $n\rightarrow\infty$
\begin{equation}
\label{eq:var gen lin stat}
\var(\length^{\varphi}(f_{n})) = c(\varphi) \cdot \log{n}+O_{\varphi}(1),
\end{equation}
where
\begin{equation*}
c(\varphi):=\frac{\|\varphi \|_{L^{2}(\sphere)}^2 }{128 \pi} > 0.
\end{equation*}
\end{theorem}

The characteristic function $\chi_{F}$ of a subsurface $F\subseteq\sphere$ with $C^{2}$ boundary
is of bounded variation i.e. $\chi_{F}\in BV(\sphere)$, \cite{GS} Example 1.4. Therefore, in this case the statement
of Theorem \ref{thm:var gen lin stat} is valid for $\length^{F}$, as the following corollary states.

\begin{corollary}\fixme{Changed the leading constant}
\label{cor:var ZF}
Let $F\subseteq \sphere$ be a subsurface of the sphere with $C^{2}$ boundary.
Then as $n\rightarrow\infty$
\begin{equation}
\label{eq:var subset len}
\var(\length^{F}(f_{n})) = c \cdot \log{n}+O_{F}(1),
\end{equation}
\begin{equation*}
c=c(F):=\frac{|F| }{128 \pi} > 0,
\end{equation*}

\end{corollary}

\begin{remark}
One may observe from the proof of Theorem \ref{thm:var gen lin stat},
that the constant involved in the ``$O$"-notation in
\eqref{eq:var gen lin stat} depends only on $\| \varphi \| _{\infty}$ and the total variation $V(\varphi)$.
In particular, the constant involved in the ``$O$"-notation \eqref{eq:var subset len} depends only on
the length of the boundary $\partial F$.
\end{remark}

\subsection{Discussion}
\label{sec:discussion}

\subsubsection{``Berry's Cancelation Phenomenon"}
\label{sec:Berry Canc Ph}

Originally, it was conjectured that the variance $\var(\length(f_{n}))$ should be asymptotic to
$c\cdot n$, where $c>0$ is a constant, due to the natural scaling; however, it turned out that $c$ vanishes,
precisely as predicted by Berry ~\cite{Berry 2002}
for the RWM. The reason for this phenomenon, which we refer to as ``Berry's cancelation phenomenon",
is that the leading nonconstant term in the long range asymptotics of the $2$-point correlation
function is purely oscillating (see the Key Proposition \ref{prop:K asymp exp}),
so that it does not contribute to the variance. The non-oscillating leading terms cancel (which is,
according to Michael Berry ~\cite{Berry 2002}, ``obscure"). It seems that ``Berry's cancelation phenomenon" is of general nature:
it also occurs on the torus ~\cite{KW}, and it is likely to hold for random combinations
of eigenfunctions on a generic manifold ~\cite{TW2}.

\subsubsection{Spherical Harmonics vs. RWM}

The principal result of the present paper shows that the behaviour of the nodal lines of $2$-dimensional
spherical harmonics of eigenvalue $E$ is consistent with the RWM of wavenumber $\sqrt{E}$,
predicted for nodal lines of generic chaotic systems.
In both cases, the expected nodal length is of order $\sqrt{E}$ and variance of order $\log{E}$. More precisely,
Berry ~\cite{Berry 2002} argued that for a billiard of area $A$, the variance of the nodal length should be asymptotic to
\begin{equation*}
\frac{A}{512\pi}\log{E}
\end{equation*}
in the high energy limit.
Taking into account the symmetry of the nodal lines on $\sphere$, the statement of Theorem 
\ref{thm:var length} is consistent to Berry's prediction.
\fixme{Revised the formulation}.

There is a direct relation between the random spherical harmonics and the RWM. Kolmogorov's theorem implies that
a random centered Gaussian ensemble of functions is determined by its covariance function (see section
\ref{sec:cov func}). The covariance function for the RWM is
\begin{equation*}
r_{RWM}(x,y) = J_{0}(\sqrt{E}|x-y|),
\end{equation*}
$x,y\in\R^{2}$, and for the random spherical harmonics is
\begin{equation*}
r_{n}(x,y) = P_{n}(\cos(d(x,y))),
\end{equation*}
where $P_{n}$ are the Legendre polynomials.
The Legendre polynomials admit Hilb's asymptotics
\begin{equation*}
P_{n}(\cos(\phi)) \approx \sqrt{\frac{\phi}{\sin\phi}}J_{0}(\phi(n+1/2)),
\end{equation*}
i.e. almost identical to RWM, up to the  ``correction factor" $\sqrt{\frac{\phi}{\sin\phi}}$.

\fixme{Omitted the a couple of sentences}.

\subsubsection{Nodal Set vs. Level Sets}

Interestingly, the behaviour of the level curves $f_{n}^{-1}(L)$ for $L>0$ is very different. Let $\length^{L}(f_{n})$ be
the length of the level curve of $f_{n}$. The expected length is ~\cite{W2}
\begin{equation*}
\E[\length^{L}(f_{n})] = c_{1}e^{-L^2/2}\sqrt{E_{n}}
\end{equation*}
consistent with the nodal case $L=0$. However, unlike the nodal lines, the variance of the level curves
length is asymptotic to \cite{W2}
\begin{equation*}
\var(\length^{L}(f_{n})) \sim c_{2}L^4 e^{-L^2}\cdot n.
\end{equation*}

\subsubsection{Real vs. Complex Zeros}
\footnotemark
\label{sec:real vs compx zeros}

\footnotetext{The author wishes to thank Mikhail Sodin and Steve Zelditch for pointing out
the unexpected differences between the real and the complex analytic cases.}

The behaviour of the zeros of {\em complex analytic} functions was studied extensively in the recent years and
it is interesting to learn that their behaviour is very different from our case of {\em real valued}
spherical harmonics. Sodin and Tsirelson ~\cite{ST} considered $3$ different models of random
complex analytic functions $\psi_{L}:\mathcal{M}\rightarrow\mathbb{C}$, all parametrized by an integer $L \rightarrow\infty$,
roughly corresponding to the degree of the harmonic polynomials $n$.
Here $\M$ is the natural domain corresponding to
the model with $G$-invariant measure $m^{*}$, where $G$ is a group of symmetries;
$\M$ is either the sphere $\mathbb{C}\cup \{ \infty\}$,
the complex plain $\mathbb{C}$ or the unit disc $\{|z|<1\}$.
In this case the set of zeros is almost surely finite.
The authors establish the asymptotic Gaussianity for
{\em smooth} linear statistics $h:\M\rightarrow\R$, $h\in C_{c}^{2}(\M)$
\begin{equation*}
\length^{h}(\psi_{L}) := \sum\limits_{z:\psi_{L}(z)=0}  h(z),
\end{equation*}
where the expected value is given by
\begin{equation*}
\E[\length^{h}(\psi_{L})] = L\cdot \frac{1}{\pi} \int\limits_{\M} h dm^{*} ,
\end{equation*}
for each of the models considered, consistent with \eqref{eq:Elen=int(phi)*sqrt(eigval)}.
However the variance is of order
\begin{equation}
\label{eq:TS var 1/L * Delta h}
\var(\length^{h}(\psi_{L})) \sim \frac{\kappa}{L}\| \Delta^{*} h \|_{L^{2}(m^*)}^{2}
\end{equation}
decaying with $L\rightarrow\infty$; here $\kappa>0$ is a universal constant, and $\Delta^{*}$
is the invariant Laplacian. Note also that here the dependency on the test function
$h$ is via the $L^2$ norm of a second order differential operator acting on $h$ (namely the invariant
Laplacian), whereas
in the real valued spherical harmonics case it depends on the $L^2$ norm of $\varphi$ itself
(i.e. the operator is the identity, see \eqref{eq:var(len^phi) cont diff asymp}
and \eqref{eq:c(phi) def}).

For $h=\chi_{U}$ the characteristic function of a smooth domain $U \subseteq \M$ (i.e. $\length^{h}$ is
the number of zeros in $U$), while the expected value of $\length^{h}(\psi_{L})$
is still proportional to $area(U)\cdot L$ the variance is of different shape
(cf. Corollary \ref{cor:var ZF} in the spherical harmonics case). Namely
it is known ~\cite{FH} that the variance is asymptotically proportional to
$\sqrt{L} \cdot \len(\partial U)$, different from Corollary \ref{cor:var ZF} both in the power of $L$
and the dependency on the test function.
This reflects the high frequency oscillations of the zeros smoothed out by a smooth test
function.

Shiffman and Zelditch ~\cite{SZ1,SZ2} considered a more general situation of
random independent Gaussian sections $s_{1}=s^{\mathcal{L},L}_{1},\ldots,
s_{k}=s^{\mathcal{L},L}_{k} \in \Gamma(\mathcal{L}^L,\M)$ of high powers
$\mathcal{L}^{L }$, $L\rightarrow\infty$ of holomorphic line bundles $\mathcal{L}$
on an $m$-dimensional K\"{a}hler manifold $\M$, where $1\le k \le m$.
They considered the volume of the intersection of the zero sets of $s_{i}$
\begin{equation*}
\length^{U}(s_{1},\ldots s_{k}) = \vol_{2m-2k}\left( (s_{1},
\ldots s_{k})^{-1}(0) \right)
\end{equation*}
and its smooth linear statistics
\begin{equation*}
\length^{h}(s_{1},\ldots s_{k})=\int\limits_{(s_{1},
\ldots s_{k})^{-1}(0)} h(z)d\vol \left((s_{1},
\ldots s_{k})^{-1}(0)\right)
\end{equation*}
(here in case the system is full $k=m$, the volume is the number of points, and the integral is a sum).

In both cases the expected value is asymptotic to
\begin{equation*}
\E\left[\length^{U}(s_{1}^{L},\ldots s^{L}_{k})\right] ,\,
\E\left[\length^{h}(s_{1}^{L},\ldots s^{L}_{k})\right] \sim c L^{k},
\end{equation*}
where as earlier, $c>0$ is proportional to either $\vol(U)$ or the mass of $h$.
For the ``sharp'' random variable they obtained ~\cite{SZ1} the asymptotic
\begin{equation*}
\var(\length^{U}(s_{1}^{L},\ldots s^{L}_{k})) \sim c_{mk} L^{2k-m-1/2} \cdot \vol(\partial U),
\end{equation*}
where $c_{mk}$ are some universal constants,
extending Forrester-Honner ~\cite{FH}, whereas for the smooth statistics they established a Central Limit Theorem with
variance
\begin{equation*}
\var(\length^{h}(s_{1}^{L},\ldots s^{L}_{k})) \sim c_{h}L^{2k-m-2},
\end{equation*}
where, as in case of Sodin-Tsirelson \eqref{eq:TS var 1/L * Delta h},
$c_{h}$ involves a certain $2$nd order differential operator acting on $h$.

\subsection{On the proof of the main results}

The proof of the Theorem \ref{thm:var length} involves some geometric as well as some probabilistic aspects;
we improve upon both in comparison with our previous paper. We employ the Kac-Rice formula, which reduces
the computation of the length variance to the $2$-point correlation function, given in terms of
distribution of the values $f_{n}(x)$ as well as their gradients $\nabla f_{n}(x)\in T_{x}(\sphere)$,
for all $x\in \sphere$.

Thanks to the isotropicity of the model,
it is sufficient to evaluate the $2$-point correlation function only on the arc $\{ \theta=0 \}$
(in the usual spherical coordinates); this reduces the problem to an essentially $1$-dimensional one.
One then has to identify the spaces $T_{x}(\sphere)$ via a family of isometries $\phi_{x}$, smooth
w.r.t. $x$, for $x$ on the arc only, which is natural in the spherical coordinates.
Scaling the arc we find out that for typical $x, y\in\sphere$, the distribution of the values and the gradients
is a small perturbation
of standard Gaussian i.i.d random variables $N(0,I)$, the latter recovering the square of the expected
value of the nodal length to be canceled. We then expand the $2$-point correlation function into a Taylor polynomial around
the asymptotic one; to do so we use Berry's elegant method.

It turns out that the long range behaviour of the two-point correlation function given is also sufficient
to extend the result to continuously differentiable linear statistics (i.e. Theorem \ref{thm:var smt lin stat}). In
the course of generalizing the proof to include this case we naturally encounter an auxiliary function
$W^{\varphi}:[0,\pi]\rightarrow\R$. To conclude the proof of Theorem \ref{thm:var smt lin stat} we will have
to understand its behaviour at the origin.

To prove Theorem \ref{thm:var gen lin stat}, we apply a standard density argument, approximating
$\varphi$ with $C^{\infty}$ functions, to which we apply Theorem \ref{thm:var smt lin stat}.
To this end we use the full strength of the statement of Theorem \ref{thm:var smt lin stat}
applied to $\varphi_{i}$, which enables to uniformly control the error term in \eqref{eq:var(len^phi) cont diff asymp}.
For a more detailed explanation see section \ref{sec:on proof Thm BV}.

\subsection{Plan of the paper}

The goal of section \ref{sec:expl Kac Rice} is to give a formula for the length variance, explicit as possible,
starting from the classical Kac-Rice formula. In section \ref{sec:asymp var}, we use the formula obtained to analyze
the variance, asymptotically for high energy (i.e. prove Theorem \ref{thm:var length}).
In sections \ref{sec:smt stat proof} and \ref{sec:gen stat proof}, we
give the proofs for Theorem \ref{thm:var smt lin stat} and Theorem \ref{thm:var gen lin stat}, respectively.

Appendix \ref{apx:covar mat} will carry on a certain technical computation we will encounter in this paper,
namely, that of covariance matrix of a random vector involving values and gradients of $f_{n}$. Appendix \ref{apx:legendre pol}
will be devoted to the Legendre polynomials and some of their basic properties. The goal of Appendix \ref{apx:BV functions}
is to give the definition and some properties of the class $BV(\sphere)$ of functions of bounded variation, including
their traces on smooth curves.

\subsection{Acknowledgement}

The author wishes to express his deepest gratitude to Zeev Rudnick, Mikhail Sodin,
St\'{e}phane Nonnenmacher, Manjunath Krishnapur and Steve Zelditch
for many fruitful and stimulating conversations that inspired this research and pushed it forward.
I wish to especially thank Dan Mangoubi for all his help and patience.
While conducting the research the author benefited from the expertise and experience of
Pengfei Guan, Dmitry Jakobson, Iosif Polterovich, John Toth and Domenico Marinucci that
proved invaluable. The author is grateful to
Sherwin Maslowe for proofreading this paper. Finally, the author would like to thank the anonymous
referee for many useful comments.

\section{An explicit integral formula for the variance}

\label{sec:expl Kac Rice}

In this section, culminating in Proposition \ref{prop:VarZ expl scal}, we derive an ``explicit"
integral formula for the variance.
First we need to introduce the covariance function.

\subsection{Covariance function}
\label{sec:cov func} The {\em covariance}
function (sometimes also referred to as two-point function) is defined as
\begin{equation}
\label{eq:u(x,y) def} u_{n}(x,y) :=\E \big[ f_{n}(x)f_{n}(y)\big]
= \frac{|\sphere|}{\eigspcdim_{n}}\sum\limits_{k=1}^{\eigspcdim_{n}}\eta_{k}^{n}
(x) \eta_{k}^{n} (y).
\end{equation}
It follows from the Kolmogorov theorem ~\cite{CL}, that, in principle $u_{n}(x,y)$ {\em determines}
the centered Gaussian random field $f_{n}$, so that one can compute any property of $f_{n}$ in terms of $u_{n}$ and
its derivatives.
By the addition theorem ~\cite{AAR}, page 456, theorem 9.6.3, $u_{n}(x,y)$ has an explicit
expression as
\begin{equation}
\label{eq:u(x,y) def ult} u_{n}(x,y) = P_{n} (\cos{d(x,y)}),
\end{equation}
where $$P_{n} :[-1,1]\rightarrow\R$$ is the Legendre polynomial of degree $n$ (see
e.g. ~\cite{S}). Recall that $d(x,y)$ is the spherical distance so
that $$\cos{d(x,y)} = \langle x,y\rangle,$$ thinking of $\sphere$ as
being embedded into $\R^{3}$.

The orthogonal invariance \eqref{eq:f(Rx1...Rxl)d=f(x1,...xl)} is then equivalent to
the corresponding property of the covariance function, namely
\begin{equation}
\label{eq:rot inv u} u_{n}(Rx,Ry) = u_{n}(x,y)
\end{equation}
for every orthogonal $R\in O(3)$. In case $y$ is not
specified, we take it to be the northern pole $N\in\sphere$, that
is
\begin{equation}
\label{eq:u(x) def} u_{n}(x) := u_{n}(x,N).
\end{equation}

For every $t\in [-1,1]$, $|P_{n} (t) | \le 1$ and
$|P_{n}(t)| =1$, if and only if $t=\pm 1$. Therefore
\begin{equation}
\label{eq:u(x,y)=1 iff x=pm y} (u_{n}(x,y)=\pm 1) \Leftrightarrow (x=\pm
y),
\end{equation}
and
\begin{equation}
\label{eq:u(x)=1 iff x=N,S} (u_{n}(x)= \pm 1 ) \Leftrightarrow (x\in
\{N, S\} ),
\end{equation}
where $N$ and $S$ are the northern and the southern poles
respectively.

\subsection{Kac-Rice formulas for moments of length}

In this section we express the first couple of moments of $\length(f_{n})$ via the Kac-Rice formula.
The most general version due to Bleher-Shiffman-Zelditch ~\cite{BSZ1,BSZ2} gives an integral expression for all the moments
$k\ge 1$ of the $(m-l)$-dimensional volume of $\{\vec{\mathcal{F}}=0\}$ for ``generic" smooth vector valued random field
$$\vec{\mathcal{F}}=(\mathcal{F}_{i})_{1\le i \le l}:\mathcal{M}\rightarrow\R^{l}$$ defined on an $m$-dimensional smooth manifold $\mathcal{M}$, $1\le l\le m$.
In our previous paper ~\cite{W1} we gave an independent elementary proof for the Kac-Rice formula in the
particular case of our interest $\mathcal{M}=\mathcal{S}^{m}$, $\mathcal{F}=f^{m}_{n}$, $k=1,2$.

To present the Kac-Rice formula in our case we will need some notation.
For $x, y \in \sphere$ we define the following random vectors:
\begin{equation*}
Z_{1}^{n;x} = \left( f_{n}(x), \nabla f_{n}(x) \right) \in \R \times T_{x}(\sphere)
\end{equation*}
and
\begin{equation}
\label{eq:Z=f(x)f(y)nabla}
Z^{n;x,y}_{2} = \left( f_{n}(x), f_{n}(y)\nabla f_{n}(x) ,\nabla f_{n}(y)\right) \in \R^2 \times T_{x}(\sphere) \times T_{y}(\sphere)
\end{equation}
More generally, for $k\ge 1$ and $x_{1},\ldots,x_{k}\in\sphere$ one may define
$$Z_{k}=Z_{k}^{n;x_{1},\ldots, x_{k}} \in \R^{k}\times \prod\limits_{i=1}^{k} T_{x_{i}}(\sphere)$$ in similar fashion.
The vectors $Z_{k}$ are all centered Gaussian in the sense that for every {\em fixed} $x_{1},\ldots, x_{k}\in\sphere$ any linear functional
of $Z_{k}$ is mean zero Gaussian.
Let $$D_{k}^{n;x_{1},\ldots, x_{k}}(v_{1},\ldots, v_{k},\xi_{1},\ldots \xi_{k})$$ be the
(mean zero Gaussian) probability density function of
$Z_{k}^{n;x_{1},\ldots, x_{k}}$. The Kac-Rice formula expresses the $k$-th moment of $\length(f_{n})$ in terms of the
distributions of $Z_{k}$ only (see Lemma \ref{lem:KacRice gen}), namely $D_{k}^{n;*}$. Therefore to express the variance
(and the expected value) of $\length(f_{n})$ we will only need to study $D_{1}^{n;*}$ and $D_{2}^{n;*}$.

\begin{lemma}[~\cite{BSZ1} Theorem 2.2; ~\cite{BSZ2} Theorem 4.3; ~\cite{W1} Proposition 3.3]
\label{lem:KacRice gen}

The first two moments of the nodal length of the spherical harmonics are given by the following formulas.

\begin{enumerate}

\item Expectation:

\begin{equation}
\label{eq:KacRice exp}
\E\left[ \length(f_{n}) \right] =
\int\limits_{\sphere} \tilde{P}_{n}(x) dx,
\end{equation}
where the density of the zero set $\tilde{P}_{n}(x)$ is given by
\begin{equation*}
\tilde{P}_{n}(x) = \int\limits_{T_{x}(\sphere)}\|\xi \| D_{1}^{n;x}(0,\xi)  d\xi.
\end{equation*}

\item Second moment:

\begin{equation}
\label{eq:KacRice var}
\E\left[ \length(f_{n}) ^2\right] = \iint\limits_{\sphere\times\sphere} \tilde{K}_{n}(x,y) dx dy,
\end{equation}
where the $2$-point correlation function $\tilde{K}_{n}$ is given by
\begin{equation}
\label{eq:tildK def}
\tilde{K}_{n}(x,y) = \frac{1}{2\pi \sqrt{1-u_{n}(x,y)^2} }\iint\limits_{T_{x}(\sphere)\times T_{y}(\sphere)}
\| \xi^{x}\|\cdot \| \xi^{y} \| D_{2}^{n}(0,0,\xi^{x},\xi^{y}) d\xi^{x} d\xi^{y}.
\end{equation}

\end{enumerate}

\end{lemma}

Neuheisel ~\cite{Neuheisel} (see also ~\cite{W1}) noticed that for every $x\in\sphere$, under any isometry
$T_{x}(\mathcal{S}^{m})\cong \R^{2}$ (i.e. any choice of orthonormal basis of $T_{x}(\sphere)$),
the distribution of $Z_{1}^{n;x}$ is mean zero Gaussian with the diagonal covariance matrix
\begin{equation*}
\left(\begin{matrix} 1  \\ &\frac{\eigval}{2}I_{2}   \end{matrix} \right),
\end{equation*}
where $I_{2}$ is the $2\times 2$ identity matrix.
It is then clear that $\tilde{P}_{n}$ is $x$-independent (this also follows from the rotational independence),
and $$\tilde{P}_{n}(x) \equiv \frac{1}{\sqrt{2\pi}},$$ by a standard computation. This, together with \eqref{eq:KacRice exp},
yields \eqref{eq:Elen=*sqrt(E)} for $m=2$ and finishes the treatment of the expectation in this case\footnotemark.
Moreover, slightly modifying the proof of \eqref{eq:KacRice exp}, we obtain
\begin{equation*}
\E\left[ \length^{\varphi}(f_{n}) \right] =
\int\limits_{\sphere} \varphi(x)\tilde{P}_{n}(x) dx,
\end{equation*}
and \eqref{eq:Elen=int(phi)*sqrt(eigval)} follows.

\footnotetext{The same computation gives the result for every $m\ge 2$}

\vspace{3mm}

The goal of the remaining part of the present section, culminating with Proposition \ref{prop:VarZ expl scal},
is to make the formula \eqref{eq:KacRice var} for the second moment ``explicit" and suitable for asymptotic analysis.
The rotational invariance \eqref{eq:f(Rx1...Rxl)d=f(x1,...xl)} of our model implies that $\tilde{K}_{n}$ depends only
on the spherical distance $d(x,y)$ between $x$ and $y$ i.e. (with a slight abuse of notations)
\begin{equation}
\label{eq:tildK(x,y)=tildK(d)}
\tilde{K}_{n}(x,y) = \tilde{K}_{n}(d(x,y)),
\end{equation}
which will be used later.

\begin{remark}
One may define $\tilde{K}_{n}(x,y)$ intrinsically as
\begin{equation*}
\tilde{K}_{n}(x,y) = \frac{1}{(2\pi)\sqrt{1-u_{n}(x,y)^2}}
\E\left[\| \nabla f(x) \| \cdot\|\nabla f(y) \| | f(x)=f(y)=0 \right],
\end{equation*}
the expectation being one of the product of gradients conditioned on $f$ vanishing at $x$ and $y$.
\end{remark}

\begin{remark}
It is important to note that
the symmetry of the nodal lines w.r.t. the involution $x\mapsto -x$ (see Remark \ref{rem:lenvarphi=0 varphi odd}) implies
that
\begin{equation}
\label{eq:tildK(x,y)=tildK(x,-y)}
\tilde{K}_{n}(x,y) = \tilde{K}_{n}(x,-y).
\end{equation}
\end{remark}

The main disadvantage of the formula \eqref{eq:tildK def} is that one has to work with probability densities
defined on the tangent planes $T_{x}(\sphere)$ which depend on the point $x\in\sphere$. In principle one may consider the tangent planes
being embedded in $\R^3$. This, however, is highly inadvisable since that would result in working with singular
Gaussians supported on a plane corresponding to $T_{x}(\sphere)$. It is thus
desired to identify for every $x\in\sphere$ $$T_{x}(\sphere)\cong\R^{2}$$ via an isometry, i.e. fix
an orthonormal basis $B_{x}$ varying {\em smoothly} for $x\in\sphere$ (i.e. an orthonormal frame).
Unfortunately it is impossible to choose a global orthonormal frame on $\sphere$;
however one can still get around that by noting that in fact all we need is a local choice for given $x, y\in\sphere$.

In general, the orthonormal frame chosen will affect the probability density function
$D_{2}^{n;*}$ of $Z_{2}^{n;*}$ induced on $\R^{6}$
(though $D^{n;*}_{1}$ will stay invariant); in section \ref{sec:Kac-Rice coord} we show how to compute $D^{n;*}_{2}$ for a given
choice of local orthonormal frames. In section \ref{sec:spher orth frames} we will show how to choose the
orthonormal frames to simplify the computations; we will use this construction while evaluating the two-point correlation
function \eqref{eq:tildK def}.

\subsection{Kac-Rice formula in coordinate system}

\label{sec:Kac-Rice coord}

Given $x, y\in\sphere$, we consider two local orthonormal frames
$F^{x}(z) = \{ e^{x}_{1}, e^{x}_{2} \}$ and
$F^{y}(z) = \{ e^{y}_{1}, e^{y}_{2} \}$, defined in
some neighbourhood of $x$ and $y$ respectively. This gives rise to (local) identifications
\begin{equation}
\label{eq:Tx,Ty=R2}
T_{x}(\sphere) \cong \R^{2} \cong T_{y}(\sphere),
\end{equation}
which are isometries.

Under the identification \eqref{eq:Tx,Ty=R2} the random vector \eqref{eq:Z=f(x)f(y)nabla}
is a $\R^{6}$ mean zero Gaussian with covariance matrix
\begin{equation*}
\Sigma=\Sigma(x,y) =  \left(\begin{matrix}A & B \\ B^{t} & C
\end{matrix}\right),
\end{equation*}
where
\begin{equation}
\label{eq:A blk def} A_{2\times 2}=A_{n}(x,y) =  \left(\begin{matrix} 1 &u_{n}(x,y) \\ u_{n}(x,y)
&1 \end{matrix}\right),
\end{equation}
\begin{equation}
\label{eq:B blk def} B_{2\times 4}=B_{n}(x,y) =  \left(\begin{matrix} \vec{0} &\nabla_{y} u_{n}(x,y) \\
\nabla_{x} u_{n}(x,y) &\vec{0} \end{matrix}\right)
\end{equation}
and
\begin{equation}
\label{eq:C blk def} C_{4\times 4}=C_{n}(x,y) =  \left(\begin{matrix} \funcdercorr I_{\spheredim} &H \\
H^{t} &\funcdercorr I_{\spheredim} \end{matrix}\right)
\end{equation}
with ``pseudo-Hessian''
\begin{equation}
\label{eq:pseudo-Hessian of u} H_{2,2}(x,y) = \left(\nabla_x \otimes \nabla_y\right) u_{n}(x,y),
\end{equation}
i.e. $H=(h_{jk})_{j,k=1,2}$ with entries given by
\begin{equation*}
h_{jk} = \frac{\partial^2}{\partial e_{j}^{x}\partial e_{k}^{y}}
u_{n}(x,y).
\end{equation*}
The covariance matrix of the Gaussian distribution of
$Z_{2}$ in \eqref{eq:Z=f(x)f(y)nabla} conditioned upon $f_{n}(x)=f_{n}(y)=0$ is given by\footnotemark
\begin{equation}
\label{eq:Omega def gen frames}
\Omega_{n}(x,y) = C - B^{t}A^{-1}B.
\end{equation}

\footnotetext{This is the inverse of the lower right corner of $\Sigma^{-1}$}

We then have a frame-dependent formula for the two-point correlation function \eqref{eq:tildK def}
\begin{equation*}
\begin{split}
\tilde{K}_{n}(x,y) &=\frac{1}{\sqrt{1-u_{n}(x,y)^2}}  \iint\limits_{\R^{2}\times\R^2}
\|w_{1} \| \cdot \|w_{2} \| \times \\&\times
\exp\left( -\frac{1}{2}(w_{1},w_{2})\Omega_{n}(x,y)^{-1} (w_{1},w_{2})^{t} \right)
\frac{dw_{1}dw_{2}}{(2\pi)^3\sqrt{\det{\Omega_{n}(x,y)}}},
\end{split}
\end{equation*}
where $\Omega_{n}(x,y)$ is given by \eqref{eq:Omega def gen frames}.

\begin{remark}
Note that even though $\tilde{K}_{n}$ is rotational invariant (i.e. $\tilde{K}_{n}(x,y)$ depends only
on the spherical distance $d(x,y)$), the same is, in general, false
for the covariance matrices $\Sigma_{n}$ (and $\Omega_{n}$).
\end{remark}

Let $\phi,\theta$ be the standard spherical coordinates on $\sphere$.
Using the rotational invariance \eqref{eq:tildK(x,y)=tildK(d)} of the $2$-point correlation function we obtain
\begin{equation*}
\begin{split}
\E\left[ \length(f_{n}) ^2\right] &= \iint\limits_{\sphere\times\sphere} \tilde{K}_{n}(x,y) dx dy
= |\sphere | \int\limits_{\sphere} \tilde{K}_{n}(N,x) dx \\&= 2\pi|\sphere |
\int\limits_{0}^{\pi} \tilde{K}_{n}(N,x(\phi)) \sin{\phi} d\phi,
\end{split}
\end{equation*}
where $x(\phi)\in\sphere$ is the point corresponding to the spherical coordinates $(\phi,\theta=0)$.
Note that $\tilde{K}(N,x(\phi))=\tilde{K}(x,y)$ for any $x,y\in\sphere$ with $d(x,y)=\phi$. We therefore have
the following corollary.

\begin{corollary}
\label{cor:EZ^2 spher coord}
One has
\begin{equation}
\label{eq:2nd mom spher coord}
\E\left[ \length(f_{n}) ^2\right] = 2\pi |\sphere |\int\limits_{0}^{\pi} \tilde{K}_{n}(\phi) \sin{\phi} d\phi,
\end{equation}
where
\begin{equation*}
\tilde{K}_{n}(\phi) = \tilde{K}_{n}(x,y),
\end{equation*}
$x,y\in\sphere$ being any pair of points with $d(x,y)=\phi$.
\end{corollary}

The main goal of the present paper is to understand the asymptotic behaviour of the function $\tilde{K}_{n}(\phi)$. To this end
we will have to provide a more explicit formula for $\tilde{K}_{n}$ by choosing concrete orthonormal frames
in section \ref{sec:spher orth frames}. It also turns out that it is more natural to {\em scale} the parameter $\phi$ by essentially
$n$; this will be done in section \ref{sec:scal 2pnt corr}.

\subsection{Choosing orthonormal frames}

\label{sec:spher orth frames}

Corollary \ref{cor:EZ^2 spher coord} implies that it is sufficient to provide a choice
$x,y\in\sphere$ with $d(x,y)=\phi$ for any given $\phi\in (0,\pi)$, and for the choice made, provide local frames
around $x$ and $y$. Hence we may restrict ourselves only to points on the half circular
arc $$\breve{NS}=\{ \theta = 0\}.$$
Let
\begin{equation}
\label{eq:orthnorm frame spher}
F=\left\{ e_{1}=\frac{\partial}{\partial \phi}, e_{2}=\frac{1}{\sin{\phi}}\frac{\partial}{\partial \theta}  \right\}
\end{equation}
be the
orthonormal frame defined on $\sphere\setminus \{N,S \}$.

Given $\phi\in (0,\pi)$ we choose any pair of points
$x,y\in\breve{NS}\setminus \{N,S\}$ with $d(x,y)=\phi$ and set $F^{x} := F$ and $F^{y}:=F$ locally in the neighbourhood of $x$ and $y$ respectively.
An explicit computation shows that in this case the covariance matrix $\Sigma_{n}(x,y)$
depends only on $\phi$ rather than on $x,y$,
and, thus so does $\Omega_{n}(\phi) = \Omega_{n}(x,y)$ of our interest. We compute
in Appendix \ref{apx:covar mat}
the conditional distribution covariance matrix explicitly to be
\begin{equation}
\label{eq:Omega(phi) def}
\Omega_{n}(\phi)=\left(\begin{matrix} \frac{E_{n}}{2}+\tilde{a} &0 &\tilde{b} &0 \\ 0 &\frac{E_{n}}{2} &0 &\tilde{c} \\
\tilde{b} &0 &\frac{E_{n}}{2}+\tilde{a} &0 \\ 0 &\tilde{c} &0 &\frac{E_{n}}{2}
\end{matrix} \right),
\end{equation}
whose entries are given by
\begin{equation}
\label{eq:tild a def}
\tilde{a}=\tilde{a}_n(\phi)= -\frac{1}{1-P_{n}(\cos{\phi})^2} \cdot P_{n}'(\cos{\phi})^2
(\sin{\phi})^2,
\end{equation}
\begin{equation}
\label{eq:tild b def}
\tilde{b}=\tilde{b}_n(\phi)=P_{n}'(\cos{\phi})\cos{\phi}-P_{n}''(\cos{\phi})(\sin{\phi})^2-
\frac{P_{n}(\cos{\phi})}{1-P_{n}(\cos{\phi})^2}\cdot
P_{n}'(\cos{\phi})^2(\sin{\phi})^2
\end{equation}
and\fixme{Changed sign in expression above}
\begin{equation}
\label{eq:tild c def}
\tilde{c}=\tilde{c}_n(\phi)=P_{n}'(\cos{\phi}).
\end{equation}

We then have
\begin{equation*}
\begin{split}
\tilde{K}_{n}(\phi) &= \iint\limits_{\R^{2}\times\R^2} \frac{1}{\sqrt{1-u(x)^2}}
\|w_{1} \| \cdot \|w_{2} \| \times \\&\times \exp\left( -\frac{1}{2}(w_{1},w_{2})\Omega_{n}(\phi)^{-1} (w_{1},w_{2})^{t} \right)
\frac{dw_{1}dw_{2}}{(2\pi)^3\sqrt{\det{\Omega_{n}(\phi)}}},
\end{split}
\end{equation*}
with the covariance matrix $\Omega_{n}$ given by \eqref{eq:Omega(phi) def}.

\begin{remark}
We choose to work with points on the arc $\{ \theta=0 \}$ since here the covariance matrix $\Omega_{n}(\phi)$
is relatively simple. This corresponds to Berry's ~\cite{Berry 2002} choice of points on the $x$ axis while dealing
with random waves on $\R^{2}$, which takes advantage of the fact that for two points $x,y \in\R^{2}$ on the $x$ axis
the canonical orthonormal bases for $T_{x}(\R^{2})$ and $T_{y}(\R^2)$ coincide under the natural identification
$T_{x}(\R^{2})\cong T_{y}(\R^2)$. Rather than working with the canonical bases, one may of course choose to work
with such orthonormal bases for {\em any} two points on the plane; this approach results in the same computation
as on the $x$ axis.
\end{remark}

\subsection{Scaling the integral formula}

\label{sec:scal 2pnt corr}

As pointed earlier, the two-point correlation function $\tilde{K}_{n}$ is expressible in terms
of the covariance function $u_{n}$ and a couple of its derivatives, which in turn are
expressible in terms of degree $n$ Legendre polynomial and its derivatives. The high energy asymptotics
$n\rightarrow\infty$ of $\tilde{K}$ is then intimately related to the behaviour of $P_{n}(\cos{d})$ for
large $n$. It is known from the Hilb's asymptotics (see Appendix \ref{apx:legendre pol}) that
\begin{equation*}
P_{n}(\cos \phi) \approx \sqrt{\frac{\phi}{\sin{\phi}}} J_{0}(\phi (n+1/2)).
\end{equation*}
It is thus only natural to introduce a new parameter $\psi$ related to $\phi$ by
\begin{equation*}
\phi = \frac{\psi}{m},
\end{equation*}
where from this point and throughout the rest of the paper we denote
\begin{equation}
\label{eq:m def}
m:=n+\frac{1}{2}.
\end{equation}
We will rewrite the formula \eqref{eq:2nd mom spher coord} in terms of $\psi$ rather than $\phi$, in hope
to simplify the subsequent computations.

\begin{proposition}
The variance of the nodal length is given by
\label{prop:VarZ expl scal}
\begin{equation}
\label{eq:var(len)=*In}
\var(\length(f_{n})) = 4\pi^2\frac{E_{n}}{n+1/2} I_{n},
\end{equation}
where
\begin{equation}
\label{eq:In def}
I_{n} = \int\limits_{0}^{\pi m} \left( K_{n}(\psi) - \frac{1}{4} \right)\sin(\psi/m) d\psi,
\end{equation}
the scaled two-point correlation function
\begin{equation}
\label{eq:Kn(psi) def}
\begin{split}
K_{n}(\psi) &= \iint\limits_{\R^{2}\times\R^2} \frac{1}{\sqrt{1-u(x)^2}}
\|w_{1} \| \cdot \|w_{2} \|\times \\&\times\exp\left( -\frac{1}{2}(w_{1},w_{2})\Delta_{n}(\psi)^{-1} (w_{1},w_{2})^{t} \right)
\frac{dz_{1}dz_{2}}{(2\pi)^3\sqrt{\det{\Delta_{n}(\psi)}}}
\end{split}
\end{equation}
with scaled covariance matrix
\begin{equation}
\label{eq:Delta def}
\Delta_{n}(\psi)= \frac{\Omega(\psi/m)}{E_{n}/2} =
\left(\begin{matrix} 1+2a &0 &2b &0 \\ 0 &1 &0 &2c \\ 2b &0 &1+2a &0 \\0 &2c &0 &1  \end{matrix}
 \right),
\end{equation}
whose entries are explicitly given by
\begin{equation}
\label{eq:a def}
\begin{split}
&a=a_n(\psi)= \frac{1}{E_n} \tilde{a}_n(\psi/m)\\&=-
\frac{1}{E_n}\frac{1}{1-P_{n}\left(\cos(\psi/m)\right)^2}
\cdot P_{n}'\left(\cos(\psi/m)\right)^2
\sin(\psi/m)^2,
\end{split}
\end{equation}
\begin{equation}
\label{eq:b def}
\begin{split}
b=b_n(\psi)&= \frac{1}{E_n} \tilde{b}_n(\psi/m)
=\frac{1}{E_n} \bigg[ P_{n}'\left(\cos(\psi/m)\right)\cos(\psi/m)
\\&-
P_{n}''\left(\cos(\psi/m)\right)\sin(\psi/m)^2\\&-
\frac{P_{n}(\cos(\psi/m))}{1-P_{n}\left(\cos(\psi/m)\right)^2}\cdot
P_{n}'\left(\cos(\psi/m)\right)^2\sin(\psi/m)^2\bigg]
\end{split}
\end{equation}
and\fixme{Changed sign in expression above}
\begin{equation}
\label{eq:c def}
c=c_n(\psi)= \frac{1}{E_n} \tilde{c}_n((\psi/m))=
\frac{1}{E_n}P_n' \left(\cos(\psi/m)\right).
\end{equation}

\end{proposition}

\begin{remark}
Using the Cauchy-Schwartz inequality one can easily check that $|b_{n}(\psi)|,|c_{n}(\psi)|\le \frac{1}{2}$.
The inequality $|a_{n}(\psi)| \le \frac{1}{2}$ is obvious.
\end{remark}

\begin{remark}
We rewrite \eqref{eq:tildK(x,y)=tildK(x,-y)} as
\begin{equation}
\label{eq:tildK(phi)=tildK(pi-phi)}
\tilde{K}_{n}(\phi) = \tilde{K}_{n}(\pi-\phi).
\end{equation}
\end{remark}

\begin{remark}
\label{rem:K(psi) prob}
One can express $K_{n}(\psi)$ in probabilistic language as
\begin{equation*}
K_{n}(\psi) = \frac{1}{(2\pi)\sqrt{1-P_{n}(\cos\psi/m)^2}}\E\left[\| U \| \cdot \| V \|       \right],
\end{equation*}
where $(U,V)$ are mean zero Gaussian random variables with covariance matrix $\Delta_{n}(\psi)$.
We will find this expression useful later, when we will study the asymptotic behaviour of $K_{n}$ for large
$\psi$ (see Proposition \ref{prop:K asymp exp}).
\end{remark}

\section{Asymptotics for the variance}
\label{sec:asymp var}

In this section we establish the asymptotics for the variance of the nodal length i.e.
prove Theorem \ref{thm:var length}. Recall that the variance of the nodal length is given
by \eqref{eq:var(len)=*In}. Thus Theorem \ref{thm:var length} is equivalent to the following
Proposition.

\begin{proposition}
\label{prop:In asymp}
As $n\rightarrow\infty$ one has
\begin{equation*}
I_{n} = \frac{1}{128\pi^2} \frac{\log{n}}{n}+O\left(\frac{1}{n}\right).
\end{equation*}
\fixme{Changed the leading constant}

\end{proposition}

The rest of the present section is dedicated to the proof of Proposition \ref{prop:In asymp}.

\subsection{Asymptotics for $I_{n}$}
\label{sec:In asympt}

Recall the definition
\begin{equation*}
I_{n} = \int\limits_{0}^{\pi m} \left( K_{n}(\psi)-\frac{1}{4}\right)\sin\left( \psi/m  \right) d\psi
\end{equation*}
of $I_{n}$, where $K_{n}$ is given by \eqref{eq:Kn(psi) def} or, equivalently, \eqref{eq:K(psi) prob},
and $m$ is related to $n$ via \eqref{eq:m def}.

We note that the scaled version of \eqref{eq:tildK(phi)=tildK(pi-phi)} is
\begin{equation}
\label{eq:Kn(psi)=Kn(pi m - psi)}
K_{n}(\psi) = K_{n}(\pi m -\psi).
\end{equation}
Thus, by the definition \eqref{eq:In def} of $I_{n}$ and \eqref{eq:Kn(psi)=Kn(pi m - psi)}, we have
\begin{equation}
\label{eq:In=2int 0-pi m/2}
I_{n} = 2\cdot\int\limits_{0}^{\pi m /2}\left( K_{n}(\psi)-\frac{1}{4}\right)\sin\left( \psi/m  \right) d\psi.
\end{equation}
Therefore we may concentrate ourselves on $[0,\pi m /2]$ rather than the full interval.

Ideally, to evaluate $I_{n}$, one would hope to obtain an explicit formula for $K_{n}(\psi)$. Unfortunately, to the best
knowledge of the author of this paper, no such formula exists. However we will still be able to give an asymptotic expression for
$K_{n}(\psi)$ for {\em large} values of $\psi$ uniformly w.r.t. $n$ and $\psi$.

For {\em small} values of $\psi$ the behaviour
of $K_{n}$ is very different, due to the fact that as $\psi\rightarrow 0+$, $P_{n}(\cos(\psi/m))$ approaches $1$, which results
in the singularity of $\frac{1}{\sqrt{1-P_{n}(\cos(\psi/m))^2}}$ and the of the covariance matrix $\Delta_{n}(\psi)$
at the origin. Nevertheless, we will see that this ``singular" contribution is negligible, so that a relatively soft upper bound
already obtained in ~\cite{W1} will suffice (see Lemma \ref{lem:sing int small}).

More precisely, we choose a constant $C>0$, which is kept fixed throughout the rest of the computations, and write
\begin{equation}
\label{eq:int 0-pi/2=int0-C+C-pi/2}
\int\limits_{0}^{\pi m /2} = \int\limits_{0}^{C}+\int\limits_{C}^{\pi m /2}.
\end{equation}
The main contribution to the integral will come from the second (``nonsingular") integral in
\eqref{eq:int 0-pi/2=int0-C+C-pi/2} i.e. outside the origin. Our first task is
then to bound the first (``singular") integral of \eqref{eq:int 0-pi/2=int0-C+C-pi/2}.
A satisfactory bound was already given in ~\cite{W1}.

\begin{lemma}[Restatement of Lemma 4.2 from ~\cite{W1}] \footnotemark
\footnotetext{Note that in ~\cite{W1}, Lemma 4.2 was given without scaling i.e. in terms of $\phi$ rather
that $\psi$}
\label{lem:sing int small}
For any constant $C>0$ we have as $n\rightarrow\infty$
\begin{equation*}
\int\limits_{0}^{C} \left| K_{n}(\psi) - \frac{1}{4} \right | \sin(\psi/m)d\psi = O\left( \frac{1}{n}   \right).
\end{equation*}
\end{lemma}

Lemma \ref{lem:sing int small} together with \eqref{eq:int 0-pi/2=int0-C+C-pi/2} and
\eqref{eq:In=2int 0-pi m/2} yield the following lemma.
\begin{lemma}
\label{lem:In=2tildIn+O()}
For any choice of the constant $C>0$ we have as $n\rightarrow \infty$
\begin{equation}
\label{eq:In=2tildIn+O()}
I_{n} = 2\tilde{I}_{n} + O\left(\frac{1}{n}\right).
\end{equation}
where
\begin{equation}
\label{eq:tildIn def}
\tilde{I}_{n} = \int\limits_{C}^{\pi m /2 } \left( K_{n}(\psi) - \frac{1}{4} \right)\sin(\psi/m) d\psi.
\end{equation}
\end{lemma}

Therefore, to understand the asymptotic behaviour of $I_{n}$ it is sufficient to understand
the asymptotic behavior of $\tilde{I}_{n}$. Proposition \ref{prop:tildIn asymp} resolves the latter.
The proof of Proposition \ref{prop:tildIn asymp} is given throughout the rest of
the present section.

\begin{proposition}
\label{prop:tildIn asymp}
For any choice of the constant $C>0$ in the definition
\eqref{eq:tildIn def} of $\tilde{I}_{n}$, we have as $n\rightarrow\infty$
\begin{equation}
\label{eq:tildIn asymp}
\tilde{I}_{n} = \frac{1}{256\pi^2} \frac{\log{n}}{n}+O\left(\frac{1}{n}\right).
\end{equation}
\fixme{Changed the leading constant}
\end{proposition}

\begin{proof}[Proof of Proposition \ref{prop:In asymp} assuming Proposition \ref{prop:tildIn asymp}]

Just use Proposition \ref{prop:tildIn asymp} together with \eqref{eq:In=2tildIn+O()}.

\end{proof}

\subsection{Asymptotics for the $2$-point correlation function}

Recall that $K_{n}(\psi)$ is given by \eqref{eq:Kn(psi) def}.
One may notice that
\begin{equation*}
K_{n}(\psi) = \frac{1}{2\pi \sqrt{1-P_{n}(\cos(\psi/m))^2}}F(a_{n}(\psi),b_{n}(\psi),c_{n}(\psi)),
\end{equation*}
where $F(\alpha,\beta,\gamma)$ is a smooth function independent of $n$, defined on some neighbourhood of
the origin $(\alpha,\beta,\gamma)=(0,0,0)$. Its arguments $a_{n}(\psi)$, $b_{n}(\psi)$
and $c_{n}(\psi)$ are {\em uniformly small} for $\psi > C$ (see Lemma \ref{lem:abc exp asymp}).

An easy explicit computation shows that
\begin{equation*}
F(0,0,0) = \frac{1}{4},
\end{equation*}
which cancels out with the constant term in \eqref{eq:In def} which corresponds to $(\E\length)^2$.
This is not a coincidence, since the origin $\alpha=\beta=\gamma=0$ corresponds to the covariance matrix $\Delta_{n}$
being the identity matrix $\Delta_{n} = I_{n}$; in this case the probability density function factors.

We choose to expand $F(\alpha,\beta,\gamma)$ into a finite Taylor polynomial around the origin\footnotemark.
We note that the matrix elements are of different order or decay rate, so that we may cut the smaller terms
earlier than the larger ones. The decay rate of $a_{n},b_{n}$ and $c_{n}$ prescribed by Lemma \ref{lem:abc exp asymp}
implies that it is sufficient to expand $a_{n}, b_{n}$ and $c_{n}$ up to $2$nd, $4$th and $1$st degrees respectively.

\footnotetext{Intuitively, the origin $a=b=c=0$ corresponds to $\psi``="\infty$ (see the decay at infinity
in Lemma \ref{lem:abc exp asymp}),
hence this expansion should be good for large values of $\psi$.}

The following is the Key Proposition for the whole paper. We will reuse it while proving Theorem \ref{thm:var smt lin stat}
(see section \ref{sec:smt stat proof}) for smooth linear statistics of the nodal line,
(see also Remark \ref{rem:changing test fnct}).

\begin{proposition}[Key Proposition]
\label{prop:K asymp exp}
For any choice of $C>0$, as $n\rightarrow\infty$, one has
\begin{equation}
\label{eq:K(psi) asympt}
\begin{split}
K_{n}(\psi) &= \frac{1}{4} + \frac{1}{2}\frac{\sin(2\psi)}{\pi n \sin(\psi/m)} +  \frac{1}{256}\frac{1}{\pi^2 n \sin(\psi/m) \psi}
+\frac{9}{32}\frac{\cos(2\psi)}{\pi n \psi \sin(\psi/m)} \\&+
\frac{\frac{27}{64}\sin(2\psi)-\frac{75}{256}\cos(4\psi)}{\pi^2 n \psi \sin(\psi/m)} +O\left( \frac{1}{ \psi^{3}}  +\frac{1}{n\psi}\right)
\end{split}
\end{equation}
uniformly\fixme{Changed constants in two places} for $C < \psi < \pi m /2$.
\end{proposition}

\begin{remark}
\label{rem:Berry canc phenom}
It is important to notice that the leading nonconstant term $\frac{1}{2}\frac{\sin(2\psi)}{\pi n \sin(\psi/m)}$ is oscillating,
and we will see that it does not contribute to the variance (see the proof of Proposition \ref{prop:In asymp}).
We observe this obscure ``Berry's cancelation phenomenon", which is responsible for the variance being surprisingly
small, in some other situations, such as Berry's original work ~\cite{Berry 2002}, and on the torus ~\cite{KW}.
This suggests that this phenomenon is of a more general nature, and we expect it to occur on a ``generic" surface ~\cite{TW2}.
\end{remark}

\begin{remark}
\label{rem:changing test fnct}
As another application of Proposition \ref{prop:K asymp exp},
one may exploit it to study the morphology of the nodal lines for $n$-dependent linear statistics $\varphi=\varphi_{n}$.
It is most efficient for $\varphi_{n}$ whose support is not shrinking too rapidly
(relatively to the scaling $\psi \approx n\phi$ we introduced earlier), for example $\varphi_{n}$ a
characteristic function of a spherical disc of radius $a_{n}$ where $$a_{n}\cdot n \rightarrow\infty.$$
For finer ``local" statistics of the nodal lines, such as studying the nodal length inside a spherical disc
of radius $\approx \frac{1}{n}$, one needs to expand $K_{n}(\psi)$ around the origin, where
the behaviour is very different from the one at infinity. We may want to do so in the future.
\end{remark}

\vspace{5mm}

The asymptotic evaluation \eqref{eq:K(psi) asympt} in Proposition \ref{prop:K asymp exp} is done
in two steps. Lemma \ref{lem:K(psi) Taylor abc} provides an approximation for the two-point correlation function
$K_{n}(\psi)$ with a polynomial in $P_{n}(\cos(\psi/m))$, $a_{n}(\psi)$, $b_{n}(\psi)$ and $c_{n}(\psi)$ i.e.
the Taylor expansion of $K_{n}$ as a function of the above expressions. In the second step, performed in Lemma
\ref{lem:abc exp asymp}, we evaluate each of the terms appearing in the Taylor expansion
obtained in the first step using the high degree asymptotics of the Legendre polynomials
and its derivatives (Hilb asymptotics).
Lemmas \ref{lem:K(psi) Taylor abc} and \ref{lem:abc exp asymp} are proved in sections \ref{sec:K(psi) Taylor abc} and
\ref{sec:abc exp asymp} respectively.

\begin{lemma}
\label{lem:K(psi) Taylor abc}
For $C>0$ large enough, one has the following expansion on $[C, \pi m /2]$
\begin{equation}
\label{eq:K(psi) Taylor abc}
\begin{split}
K_{n}(\psi) &= \frac{1}{4}+\frac{1}{4} \cdot a_{n}(\psi) + \frac{1}{8}\cdot b_{n}(\psi)^2
-\frac{1}{32}\cdot a_{n}(\psi)^2-\frac{3}{16}\cdot a_{n}(\psi)b_{n}(\psi)^2\\&+\frac{3}{128}\cdot b_{n}(\psi)^4+
\frac{1}{8}\cdot P_{n}(\cos(\psi/m))^2+ \frac{1}{8}\cdot a_{n}(\psi) P_{n}(\cos(\psi/m))^2 \\&+
\frac{1}{16}\cdot b_{n}(\psi)^2 P_{n}(\cos(\phi/m))^2 + \frac{3}{32}\cdot P_{n}(\cos(\phi/m))^4
\\&+O\left(P_{n}(\cos(\psi/m))^6 + a_{n}(\psi)^3+b_{n}(\psi)^5+c_{n}(\psi)^2\right),
\end{split}
\end{equation}
where the constants involved in the ``$O$''-notation depend only on $C$.

\end{lemma}

\begin{lemma}
\label{lem:abc exp asymp}

For $n\ge 1$, $C<\psi<\pi m/2$ we have the following estimates for $a,b,c$ and $P_{n}(\cos(\psi/m))$.

\begin{enumerate}

\item
\label{it:Pn(cos psi/m)^2 asymp}
\begin{equation}
\label{eq:Pn(cos psi/m)^2 asymp}
\begin{split}
P_{n}(\cos(\psi/m))^2 &=\frac{1+\sin(2\psi)}{\pi n \sin(\psi/m)} - \frac{\cos(2\psi)}{4\pi n\psi\sin(\psi/m)}
+O\left(\frac{1}{\psi^{3}} + \frac{1}{\psi n}\right)
\\&= \frac{1+\sin(2\psi)}{\pi n \sin(\psi/m)}+
O\left(\frac{1}{\psi^{2}}\right).
\end{split}
\end{equation}

\item
\label{it:an(psi) asymp}
\begin{equation*}
\begin{split}
a_{n}(\psi) &= -\frac{1-\sin(2\psi)}{\pi n \sin(\psi/m)} +
\frac{3\cos(2\psi)}{4\pi n^2\sin(\psi/m)^2}-\frac{1+\cos(4\psi)}{2\pi^2 n^2\sin(\psi/m)^2}
+O\left(\frac{1}{\psi^{3}}+  \frac{1}{n\psi}  \right)
\\&= -\frac{1-\sin(2\psi)}{\pi n \sin(\psi/m)}  + O\left(\frac{1}{\psi^2}  \right)
\end{split}
\end{equation*}

\item
\label{it:bn(psi)^2 asymp}

\begin{equation*}
\begin{split}
b_{n}(\psi)^2 &= \frac{1+\sin(2\psi)}{\pi n \sin(\psi/m)} +\frac{7 \cos(2\psi)}{4 \pi n \sin(\psi/m) \psi}
-\frac{1+\cos(4\psi)}{\pi^2 n\sin(\psi/m)\psi} +O\left( \frac{1}{ \psi^{3}}  +\frac{1}{n\psi}\right)
\end{split}
\end{equation*}
\fixme{Changed a sign in the previous equation}

\item
\label{it:c << psi^-3/2}
\begin{equation*}
|c_{n}(\psi)| = O\left( \frac{1}{\psi^{3/2}}   \right)
\end{equation*}

\item
\label{it:Pn(cos psi/m)^4 asymp}

\begin{equation*}
P_{n}(\cos(\psi/m))^4 = \frac{\frac{3}{2}+2\sin(2\psi)-\frac{1}{2}\cos(4\psi)}{\pi^2 n^2\sin(\psi/m)^2}
+ O\left( \frac{1}{\psi^{3}}\right)
\end{equation*}

\item

\begin{equation*}
a_{n}(\psi)^2 =  \frac{\frac{3}{2}-2\sin(2\psi)-\frac{1}{2}\cos(4\psi)}{\pi^2 n^2\sin(\psi/m)^2}
+ O\left( \frac{1}{\psi^{3}}\right)
\end{equation*}

\item

\begin{equation*}
b_{n}(\psi)^4 = \frac{\frac{3}{2}+2\sin(2\psi)-\frac{1}{2}\cos(4\psi)}{\pi^2 n^2\sin(\psi/m)^2}
+ O\left( \frac{1}{\psi^{3}}\right)
\end{equation*}

\item
\begin{equation*}
P_{n}(\cos(\psi/m))^2 a_{n}(\psi) = -\frac{1+\cos(4\psi)}{2\pi^2 n^2 \sin(\psi/m)^2}+O\left( \frac{1}{\psi^3} \right)
\end{equation*}

\item
\begin{equation*}
P_{n}(\cos(\psi/m))^2 b_{n}(\psi)^2 = \frac{3/2+2\sin(2\psi)-\frac{1}{2}\cos(4\psi)}{\pi^2 n^2 \sin(\psi/m)^2}+
O\left( \frac{1}{\psi^3}    \right)
\end{equation*}

\item
\label{it:a*Pn(cos psi/m)^2 asymp}

\begin{equation*}
a_{n}(\psi)b_{n}(\psi)^2 = -\frac{1+\cos(4\psi)}{2\pi^2 n^2 \sin(\psi/m)^2}+O\left( \frac{1}{\psi^3} \right)
\end{equation*}

\end{enumerate}

\end{lemma}

\begin{proof}[Proof of Proposition \ref{prop:K asymp exp}]

Substituting all the various estimates in Lemma \ref{lem:abc exp asymp} into
\eqref{eq:K(psi) Taylor abc} we obtain, after collecting similar terms together and some reorganization
(replacing $\frac{1}{n^2\sin(\psi/m)^2}$ by $\frac{1}{n\sin(\psi/m)\psi}$ whenever necessary)
\begin{equation*}
\begin{split}
K_{n}(\psi) &= \frac{1}{4} + \frac{1}{\pi n \sin(\psi/m)}\left( \frac{1}{4}\cdot \sin(2\psi)+\frac{1}{8}\sin(2\psi)
+\frac{1}{8}\sin(2\psi)\right)
\\&+ \frac{1}{\pi^2 n \sin(\psi/m) \psi} \bigg(  -\frac{1}{4}\cdot \frac{1}{2} -\frac{1}{8}\cdot 1
-\frac{1}{32}\cdot \frac{3}{2} +\frac{3}{16}\cdot \frac{1}{2} +\frac{3}{128}\cdot \frac{3}{2} \\&+\frac{1}{8}\cdot 0
-\frac{1}{8}\cdot \frac{1}{2} +\frac{1}{16}\cdot \frac{3}{2} +\frac{3}{32}\cdot \frac{3}{2}   \bigg)
\\&+\frac{1}{\pi n \psi \sin(\psi/m)} \left( \frac{1}{4}\cdot \frac{3}{4}\cos(2\psi)+\frac{1}{8}\cdot\frac{7}{4} \cos(2\psi)
-\frac{1}{8}\cdot \frac{1}{4}\cos(2\psi)
\right)\\&+\frac{1}{\pi^2 n \psi \sin(\psi/m)} \bigg(-\frac{1}{4}\cdot\frac{1}{2}\cos(4\psi)-\frac{1}{8}\cdot \cos(4\psi)
+\frac{1}{32}(2\sin(2\psi)+\frac{1}{2}\cos(4\psi))\\&+\frac{3}{16}\cdot \frac{1}{2}\cos(4\psi)
+\frac{3}{128}(2\sin(2\psi)-\frac{1}{2}\cos(4\psi)) -\frac{1}{8}\cdot \frac{1}{2}\cos(4\psi)
\\&+\frac{1}{16}(2\sin(2\psi)- \frac{1}{2}\cos(4\psi))
+\frac{3}{32}\cdot (2\sin(2\psi)-\frac{1}{2}\cos(4\psi))    \bigg)+O\left( \frac{1}{ \psi^{3}}  +\frac{1}{n\psi}\right)
\\&= \frac{1}{4} + \frac{1}{2}\frac{\sin(2\psi)}{\pi n \sin(\psi/m)} +  \frac{1}{256}\frac{1}{\pi^2 n \sin(\psi/m) \psi}
+\frac{3}{8}\frac{\cos(2\psi)}{\pi n \psi \sin(\psi/m)} \\&+
\frac{\frac{27}{64}\sin(2\psi)-\frac{75}{256}\cos(4\psi)}{\pi^2 n \psi \sin(\psi/m)} +O\left( \frac{1}{ \psi^{3}}  +\frac{1}{n\psi}\right).
\end{split}
\end{equation*}
\fixme{In the equation above changed a sign in the first equality ($\frac{1}{8}$ into $-\frac{1}{8}$),
the constant $-\frac{11}{256}$ into $-\frac{75}{256}$ in the second equality and
the constant in the second inequality: $\frac{65}{256}$ into $\frac{1}{64}$}

\end{proof}

\subsection{Concluding the proof of Proposition \ref{prop:tildIn asymp}}
\label{sec:tild{In} conclusion}

All the hard work establishing the asymptotics \eqref{eq:K(psi) asympt} of $K_{n}(\psi)$ at infinity finally pays off as the proof of
Proposition \ref{prop:tildIn asymp} is now straightforward.

\begin{proof}[Proof of Proposition \ref{prop:tildIn asymp}]
Recall that $\tilde{I}_{n}$ is given by \eqref{eq:tildIn def}, where $K_{n}(\psi)$ for large $\psi$ we may asymptotically expand
$K_{n}(\psi)$ as \eqref{eq:K(psi) asympt}. First note that the constant term $\frac{1}{4}$ in \eqref{eq:K(psi) asympt} cancels out in \eqref{eq:tildIn def}.
Thus we have
\begin{equation}
\label{eq:In=int(asympt)+error}
\begin{split}
\tilde{I}_{n} &= \int\limits_{C}^{\pi m/2} \bigg[
\frac{1}{2}\frac{\sin(2\psi)}{\pi n \sin(\psi/m)} +  \frac{1}{256}\frac{1}{\pi^2 n \sin(\psi/m) \psi}
+\frac{9}{32}\frac{\cos(2\psi)}{\pi n \psi \sin(\psi/m)} \\&+
\frac{\frac{27}{64}\sin(2\psi)-\frac{75}{256}\cos(4\psi)}{\pi^2 n \psi \sin(\psi/m)}\bigg]\sin(\psi/m)d\psi
+O\left( \int\limits_{C}^{\pi m /2}\left[\frac{1}{ \psi^{3}}  +\frac{1}{n\psi} \right]\sin(\psi/m)d\psi\right)
\\&= \frac{1}{\pi n}\int\limits_{C}^{\pi m/2} \bigg[
\frac{1}{2}\sin(2\psi) +  \frac{1}{256}\frac{1}{\pi \psi}
+\frac{9}{32}\frac{\cos(2\psi)}{ \psi} +
\frac{\frac{27}{64}\sin(2\psi)-\frac{75}{256}\cos(4\psi)}{\pi \psi}\bigg]d\psi
\\&+O\left( \frac{1}{n}\right).
\end{split}
\end{equation}
\fixme{Changed the constants - 4 places}

The contribution of the first term in \eqref{eq:In=int(asympt)+error} is bounded by
\begin{equation*}
\ll \frac{1}{n} \int\limits_{C}^{\pi m / 2} \sin(2\psi)d\psi =
 O\left( \frac{1}{n}  \right).
\end{equation*}

The main contribution to \eqref{eq:In=int(asympt)+error} comes from the leading non-oscillatory term i.e. the second term:
\begin{equation}
\label{eq:K(psi) main contr}
\begin{split}
\frac{1}{256 \pi^2 n}\int \limits_{C}^{\pi m /2} \frac{d\psi}{\psi}=
\frac{1}{256 \pi^2}\cdot \frac{\log{n}}{n} +O\left( \frac{1}{n}  \right).
\end{split}
\end{equation}
\fixme{Changed the constant (2 places)}

Bounding the contribution of the other oscillatory terms using integration by parts, as well as
bounding the error term in \eqref{eq:In=int(asympt)+error}, is
easy. The asymptotic expression \eqref{eq:K(psi) main contr} together with the bound for the
contribution of the other terms in \eqref{eq:K(psi) asympt} yields the result \eqref{eq:tildIn asymp} of the
present proposition.

\end{proof}

\subsection{Proofs of auxiliary lemmas}

\subsubsection{Taylor expansion for $K_{n}(\psi)$}
\label{sec:K(psi) Taylor abc}

In principle, one may compute the coefficients of the multivariate Taylor expansion directly using the Leibnitz rule for differentiating
under the integral sign. The following elegant method due to Berry ~\cite{Berry 2002} gives the necessary Taylor
coefficients avoiding the long and tedious computations.

\begin{proof}[Proof of Lemma \ref{lem:K(psi) Taylor abc}]
We use Remark \ref{rem:K(psi) prob} to write
\begin{equation}
\label{eq:K(psi) prob}
K_{n}(\psi) = \frac{1}{(2\pi)\sqrt{1-P_{n}(\cos\psi/m)^2}}\E\left[\| U \| \| V \|       \right],
\end{equation}
where ($U,V$) is a mean zero multivariate Gaussian random vector with covariance matrix $\Delta=\Delta_{n}(\psi)$
given by \eqref{eq:Delta def}. On $[C,\pi m/2]$ we may Taylor expand the first term in \eqref{eq:K(psi) prob} as
\begin{equation}
\label{eq:1/sqrt(1-Pn^2) Taylor}
\frac{1}{\sqrt{1-P_{n}(\cos\psi/m)^2}} = 1 + \frac{1}{2}P_{n}(\cos\psi/m)^2 +\frac{3}{8}P_{n}(\cos\psi/m)^4+O\left( P_{n}(\cos\psi/m)^6\right),
\end{equation}
since by part \ref{it:Pn(cos psi/m)^2 asymp} of Lemma \ref{lem:abc exp asymp}, $|P_{n}(\cos(\psi/m))|$ is bounded away from $1$,
provided that $C$ is large enough. It then remains to expand the remaining part of \eqref{eq:K(psi) prob}
i.e. $\E\left[\| U \| \| V \| \right]$ in terms of powers of $a$, $b$ and $c$.

To this end we use the identity
\begin{equation*}
\sqrt{\alpha} = \frac{1}{\sqrt{2\pi}}\int\limits_{0}^{\infty} (1-e^{-\frac{\alpha t}{2}})\frac{dt}{t^{3/2}},
\end{equation*}
which implies
\begin{equation}
\label{eq:EUV via int f}
\begin{split}
\E[\| U \| \| V\|] &= \E \left[ \frac{1}{2\pi} \int_{0}^{\infty} \int_{0}^{\infty} \left( 1-e^{-\frac{t\| U \|^2}{2}}\right)
\left( 1-e^{-\frac{t\| V \|^2}{2}}\right) \frac{dt ds}{(ts)^{3/2}} \right] \\&=
\frac{1}{2\pi} \int_{0}^{\infty} \int_{0}^{\infty} [f(0,0)-f(t,0)-f(0,s)+f(t,s)] \frac{dt ds}{(ts)^{3/2}}
,
\end{split}
\end{equation}
where we define
\begin{equation}
\label{eq:f def}
\begin{split}
f(t,s) &=f^{\Delta_{n}(\psi)}(t,s) :=\E\left[e^{-\frac{t\|U\|^2+s\| V\|^2}{2}} \right] \\&=
\frac{1}{(2\pi)^2(\det\Delta_{n}(\psi))^{1/2}}\int\limits_{\R^2\times\R^2} e^{-\frac{1}{2}W^{t}\left(\Delta_{n}(\psi)^{-1}
+\left(\begin{matrix} tI_{2} \\ &sI_{2}\end{matrix} \right) \right)W}  dW
\\&= \frac{1}{(\det\Delta_{n}(\psi))^{1/2} \det \left(\Delta_{n}(\psi)^{-1}
+\left(\begin{matrix} tI_{2} \\ &sI_{2}\end{matrix} \right) \right)^{1/2}}
\\&= \frac{1}{\det \left( I + \left(\begin{matrix} tI_{2} \\ &sI_{2}\end{matrix} \right)\Delta_{n}(\psi) \right) ^{1/2}}.
\end{split}
\end{equation}

Let
\begin{equation*}
\Delta_{n}(\psi) = I+\left(\begin{matrix} A &B \\B &A \end{matrix}\right),
\end{equation*}
where
\begin{equation*}
A=A_{n}(\psi) = \left( \begin{matrix} 2a \\&0  \end{matrix}\right) ;\quad B=B_{n}(\psi)= \left( \begin{matrix}  2b \\&2c \end{matrix}\right),
\end{equation*}
with entries defined by \eqref{eq:a def}, \eqref{eq:b def} and \eqref{eq:c def}.
Thus we have
\begin{equation*}
I + \left(\begin{matrix} tI_{2} \\ &sI_{2}\end{matrix} \right)\Delta_{n}(\psi) =
\left(\begin{matrix}  (1+t)I+tA &tB \\sB &(1+s)I+sA  \end{matrix} \right),
\end{equation*}
so that
\begin{equation}
\label{eq:det I+tsDelta}
\begin{split}
&\det\left(I + \left(\begin{matrix} tI_{2} \\ &sI_{2}\end{matrix} \right)\Delta\right)\\&=
\det \left((1+t)I+tA\right)\det \left( (1+s)I+sA - st B ((1+t)I+tA)^{-1}B\right) \\&=
(1+t)^2 (1+s)^2 \det \left(I+\frac{t}{1+t}A\right)\times\\&\times
\det \left( I+\frac{s}{1+s}A - \frac{st}{(1+s)(1+t)} B \left(I+\frac{t}{1+t}A\right)^{-1}B\right),
\end{split}
\end{equation}
where we make use of the fact that both $A$ and $B$ are diagonal and hence commute.

We compute the first determinant explicitly as
\begin{equation}
\label{eq:det I+tsDelta part1}
\det \left(I+\frac{t}{1+t}A\right) = 1+\frac{t}{1+t}2a.
\end{equation}
Next we wish to compute the other determinant in \eqref{eq:det I+tsDelta}. For this we write
\begin{equation*}
\left(I+\frac{t}{1+t}A\right)^{-1} =  I - \frac{t}{1+t}A+O\left(a^2\right)
\end{equation*}
where we understand the ``O"-notation entry-wise.
Therefore (taking advantage of the fact that all the matrices involved are diagonal), we have
\begin{equation*}
\begin{split}
&I+\frac{s}{1+s}A - \frac{st}{(1+s)(1+t)} B^2 \left(I+\frac{t}{1+t}A\right)^{-1} \\&=
I +\frac{s}{1+s}A- \frac{st}{(1+s)(1+t)} B^2+\frac{st^2}{(1+s)(1+t)^2}AB^2 +O\left(a^2b^2\right)
\\&= \left(\begin{matrix} 1+\frac{2s}{1+s}a-\frac{4st}{(1+s)(1+t)}b^2+\frac{8st^2}{(1+s)(1+t)^2}ab^2  \\
&1-\frac{4st}{(1+s)(1+t)}c^2 \end{matrix} \right)
+O\left(a^2b^2\right).
\end{split}
\end{equation*}
Therefore, we have
\begin{equation*}
\begin{split}
&\det\left(I+\frac{s}{1+s}A - \frac{st}{(1+s)(1+t)} B^2 \left(I+\frac{t}{1+t}A\right)^{-1}\right) \\&=
\left( 1+\frac{2s}{1+s}a-\frac{4st}{(1+s)(1+t)}b^2+\frac{8st^2}{(1+s)(1+t)^2}ab^2 \right)
\left( 1-\frac{4st}{(1+s)(1+t)}c^2 \right)
 \\&+ O\left( (a^2b^2+c^2)\right) \\&= \left( 1+\frac{2s}{1+s}a-\frac{4st}{(1+s)(1+t)}b^2+\frac{8st^2}{(1+s)(1+t)^2}ab^2 \right)
  + O\left( a^2b^2+c^2 \right).
\end{split}
\end{equation*}

Substituting \eqref{eq:det I+tsDelta part1} and \eqref{eq:det I+tsDelta part2} into \eqref{eq:det I+tsDelta} we obtain
\begin{equation}
\label{eq:det I+tsDelta part2}
\begin{split}
&\det\left(I + \left(\begin{matrix} tI_{2} \\ &sI_{2}\end{matrix} \right)\Delta\right)=
(1+t)^2 (1+s)^2 \left( 1+\frac{2t}{1+t}a \right) \times \\&\times
\left( 1+\frac{2s}{1+s}a-\frac{4st}{(1+s)(1+t)}b^2+\frac{8st^2}{(1+s)(1+t)^2}ab^2
+O\left( a^2b^2+c^2 \right) \right)  \\&=
(1+t)^2 (1+s)^2  \times\\&\times
\bigg( 1+\frac{2(t+s+2st)}{(1+s)(1+t)}a+\frac{4st}{(1+s)(1+t)}(a^2-b^2)+ O\left(  a^2b^2+c^2 \right) \bigg) .
\end{split}
\end{equation}
Now we use the expansion
\begin{equation*}
\frac{1}{\sqrt{1+x}} = 1 - \frac{1}{2}x+\frac{3}{8}x^2+O(x^3)
\end{equation*}
to write
\begin{equation}
\label{eq:f asymp exp}
\begin{split}
&f^{\Delta_{n}(\psi)}(t,s)= \frac{1}{\det\left(I + \left(\begin{matrix} tI_{2} \\ &sI_{2}\end{matrix} \right)\Delta\right)^{1/2}}=
\frac{1}{(1+t) (1+s)}  \times\\&\times
\bigg( 1-\frac{(t+s+2st)}{(1+s)(1+t)}a+\frac{2st}{(1+s)(1+t)}b^2+\left(\frac{3s^2}{2(1+s)^2}+\frac{3t^2}{2(1+t)^2}
 +\frac{st}{(1+s)(1+t)}\right)a^2\\&-\frac{6st}{(1+t)^2(1+s)^2} (t+s+2st)ab^2 + \frac{6t^2s^2}{(1+t)^2(1+s)^2}b^4+ O(a^3+b^5+c^2) \bigg) .
\end{split}
\end{equation}

The asymptotic expansion \eqref{eq:f asymp exp} implies, in particular,
\begin{equation}
\label{eq:f(0,0) asymp exp}
f^{\Delta_{n}(\psi)}(0,0) = 1,
\end{equation}
\begin{equation}
\label{eq:f(t,0) asymp exp}
f^{\Delta_{n}(\psi)}(t,0) = \frac{1}{1+t} \left(1-\frac{t}{1+t}a+\frac{3t^2}{2(1+t)^2}a^2 + O(a^3+b^5+c^2) \right),
\end{equation}
and
\begin{equation}
\label{eq:f(0,s) asymp exp}
f^{\Delta_{n}(\psi)}(0,s) = \frac{1}{1+s} \left(1-\frac{s}{1+s}a+\frac{3s^2}{2(1+s)^2}a^2 + O(a^3+b^5+c^2) \right)
\end{equation}

Define
\begin{equation}
\label{eq:F def}
\begin{split}
F^{\Delta_{n}(\psi)}(t,s):=f^{\Delta_{n}(\psi)}(0,0)-f^{\Delta_{n}(\psi)}(t,0)-f^{\Delta_{n}(\psi)}(0,s)+f^{\Delta_{n}(\psi)}(t,s),
\end{split}
\end{equation}
so that in the new notations \eqref{eq:EUV via int f} is
\begin{equation}
\label{eq:EUV via int F}
\E[\| U \| \| V\|] =
\frac{1}{2\pi} \int_{0}^{\infty} \int_{0}^{\infty} F^{\Delta_{n}(\psi)}(t,s) \frac{dt ds}{(ts)^{3/2}}.
\end{equation}
Plugging the estimates \eqref{eq:f asymp exp}, \eqref{eq:f(0,0) asymp exp}, \eqref{eq:f(t,0) asymp exp} and \eqref{eq:f(0,s) asymp exp}
into the definition \eqref{eq:F def} of $F^{\Delta_{n}(\psi)}$ yields
\begin{equation}
\label{eq:F(t,s) Taylor a,b,c}
\begin{split}
F^{\Delta_{n}(\psi)}(t,s)&= \frac{ts}{(1+t)(1+s)} + \frac{st(2+t+s)}{(1+t)^2(1+s)^2}\cdot a \\&+
\frac{2st}{(1+s)^2(1+t)^2}\cdot b^2 -\frac{ts(t+10ts+s+3ts^2+3t^2s  -2)}{2(1+t)^3(1+s)^3}\cdot a^2
\\&-\frac{6st}{(1+t)^3(1+s)^3} (t+s+2st) \cdot ab^2+\frac{6t^2s^2}{(1+t)^3(1+s)^3}\cdot b^4
+O(a^3+b^5+c^2),
\end{split}
\end{equation}
where the constants involved in the $``O"$-notation are {\em universal}. We wish to plug \eqref{eq:F(t,s) Taylor a,b,c}
into \eqref{eq:EUV via int F} and integrate with respect to $t$ and $s$. The problem is that the integral $\int\limits_{0}^{\infty} \frac{dt}{t^{3/2}}$
diverges at the origin so that the bound for the error term in \eqref{eq:F(t,s) Taylor a,b,c} is not sufficient.
To resolve this isssue we notice that \eqref{eq:F def} implies that we have
\begin{equation}
\label{eq:F(0,s)=F(t,0)=0}
F^{\Delta_{n}(\psi)}(t,s)|_{t=0}=F^{\Delta_{n}(\psi)}(t,s)|_{s=0} = 0.
\end{equation}
and identify the expression \eqref{eq:F(t,s) Taylor a,b,c} as the Taylor expansion of $F^{\Delta_{n}(\psi)}(t,s)$ considered
as a function of $(a,b,c)$ with fixed parameters $t,s$ around the origin
$(a,b,c)=(0,0,0)$. The vanishing
property \eqref{eq:F(0,s)=F(t,0)=0} implies that all the Taylor coefficients in the expansion \eqref{eq:F(t,s) Taylor a,b,c}
considered as a function of $t,s$, are divisible by $ts$, so
that we may improve the error term in \eqref{eq:F(t,s) Taylor a,b,c}  as
\begin{equation}
\label{eq:F(t,s) Taylor imprv error}
\begin{split}
F^{\Delta_{n}(\psi)}(t,s)&= \frac{ts}{(1+t)(1+s)} + \frac{st(2+t+s)}{(1+t)^2(1+s)^2}\cdot a \\&+
\frac{2st}{(1+s)^2(1+t)^2}\cdot b^2 -\frac{ts(t+10ts+s+3ts^2+3t^2s  -2)}{2(1+t)^3(1+s)^3}\cdot a^2
\\&-\frac{6st}{(1+t)^3(1+s)^3} (t+s+2st) \cdot ab^2+\frac{6t^2s^2}{(1+t)^3(1+s)^3}\cdot b^4
\\&+O\left( m(t,s)\cdot (a^3+b^5+c^2)\right),
\end{split}
\end{equation}
where we introduce the notations $m(t):=\min\{ t,1\} $ and $$m(t,s):=m(t)\cdot m(s) .$$ Plugging
\eqref{eq:F(t,s) Taylor imprv error} into \eqref{eq:EUV via int F} and integrating term by term we obtain
\begin{equation}
\label{eq:EUV Taylor abc}
\begin{split}
\E [\|U\| \| V \| ] =  \frac{\pi}{2}
+\frac{\pi}{2} \cdot a + \frac{\pi}{4}\cdot b^2-\frac{\pi}{16}\cdot a^2-\frac{3\pi}{8}\cdot ab^2+\frac{3\pi}{64}b^4+
O(a^3+b^5+c^2),
\end{split}
\end{equation}
where we used the standard integrals
\begin{equation*}
\int\limits_{0}^{\infty} \frac{dt}{\sqrt{t}(1+t)} = \pi , \; \int\limits_{0}^{\infty} \frac{dt}{\sqrt{t}(1+t)^2} = \frac{\pi}{2} ,\;
\int\limits_{0}^{\infty} \frac{\sqrt{t} }{(1+t)^2}dt = \frac{\pi}{2},
\end{equation*}
\begin{equation*}
\int\limits_{0}^{\infty} \frac{t^{3/2}}{(1+t)^3}dt = \frac{3\pi}{8}, \; \int\limits_{0}^{\infty} \frac{\sqrt{t}}{(1+t)^3}dt = \frac{\pi}{8},
\; \int\limits_{0}^{\infty} \frac{dt}{\sqrt{t}(1+t)^3} = \frac{3\pi}{8}.
\end{equation*}

We finally plug the estimates \eqref{eq:1/sqrt(1-Pn^2) Taylor} and \eqref{eq:EUV Taylor abc} into
\eqref{eq:K(psi) prob} to obtain \eqref{eq:K(psi) Taylor abc}, that is the statement of the present Lemma.
\end{proof}

\subsubsection{Some estimates related to the matrix elements}

In this section we evaluate the various expressions appearing in \eqref{eq:K(psi) Taylor abc}
asymptotically as $\psi\rightarrow\infty$, namely prove Lemma \ref{lem:abc exp asymp}.
To evaluate the matrix elements $a,b,c$ we will need to deal with the asymptotic behaviour
of the Legendre polynomials of high degree. The reader may find the necessary background on
the Legendre polynomials as well as some basic asymptotic estimates
in Appendix \ref{apx:legendre pol} (see Lemma \ref{lem:Pn,Pn',Pn'' asymp est}).

\label{sec:abc exp asymp}

\begin{proof}[Proof of Lemma \ref{lem:abc exp asymp}]

It is easy to check that parts \ref{it:Pn(cos psi/m)^4 asymp}-\ref{it:a*Pn(cos psi/m)^2 asymp} of the present lemma follow directly from
parts \ref{it:Pn(cos psi/m)^2 asymp}-\ref{it:c << psi^-3/2}. Moreover, part \ref{it:Pn(cos psi/m)^2 asymp} may be obtained
by a straightforward application of \eqref{eq:Pn asymp}, and part \ref{it:c << psi^-3/2} is a direct consequence
of the high degree asymptotics \eqref{eq:Pn' asymp} for the derivatives of Legendre polynomials. It then remains to prove
parts \ref{it:an(psi) asymp}-\ref{it:bn(psi)^2 asymp}.

Recall they we assume that $C < \psi < \pi m /2$, so that $P_{n}(\cos(\psi/m))$ is bounded away from $1$
by Hilb's asymptotics \eqref{eq:Hilb asymp gen}. Hence we may write
\begin{equation}
\label{eq:Legendre der sum}
\begin{split}
a_{n}(\psi) &= -\frac{1}{n^2}P'(\cos(\psi/m))^2\sin(\psi/m)^2 -\frac{1}{n^2}P(\cos(\psi/m))^2 P'(\cos(\psi/m))^2\sin(\psi/m)^2
\\&+O\left( \frac{1}{n\psi}+  \frac{1}{\psi^{3}}  \right),
\end{split}
\end{equation}
where to bound the error term we used the decay
\begin{equation}
\label{eq:Pn' decay}
|P'(\cos(\psi/m))| = O\left( \frac{n^2}{\psi^{3/2}}  \right),
\end{equation}
which follows from \eqref{eq:Pn' asymp}.

Now we use \eqref{eq:Pn' asymp} to obtain
\begin{equation}
\label{eq:1/n^2Pn'^2sin(psi/m)^2 asymp}
\begin{split}
&\frac{1}{n^2}P_{n}'(\cos(\psi/m))^2\sin(\psi/m)^2=\frac{2}{\pi n \sin(\psi/m)^3}
\bigg( \sin(\psi/m)^2\sin\left(\psi-\frac{\pi}{4}   \right)^2 \\&-\frac{3}{8 n}\sin(\psi/m)\cos(2\psi)+O\left( \frac{1}{n^2} \right)  \bigg)
+O\left(\frac{1}{\psi^{3}}+  \frac{1}{n\psi}  \right)
\\&= \frac{1-\sin(2\psi)}{\pi n \sin(\psi/m)} -\frac{3\cos(2\psi)}{4 \pi n^2 \sin(\psi/m)^2}
+O\left(\frac{1}{\psi^{3}}+  \frac{1}{n\psi}  \right),
\end{split}
\end{equation}
and \eqref{eq:Pn(cos psi/m)^2 asymp} together with \eqref{eq:1/n^2Pn'^2sin(psi/m)^2 asymp} imply
\begin{equation}
\label{eq:1/n^2Pn^2Pn'^2sin(psi/m)^2 asymp}
\begin{split}
&\frac{1}{n^2}P_{n}(\cos(\psi/m))^2 P_{n}'(\cos(\psi/m))^2\sin(\psi/m)^2=\frac{1+\cos(4\psi)}{2\pi^2 n^2 \sin(\psi/m)^2}+
O\left(\frac{1}{\psi^{3}}+  \frac{1}{n\psi}  \right).
\end{split}
\end{equation}
Substituting \eqref{eq:1/n^2Pn'^2sin(psi/m)^2 asymp} and \eqref{eq:1/n^2Pn^2Pn'^2sin(psi/m)^2 asymp}
into \eqref{eq:Legendre der sum} we obtain part \ref{it:an(psi) asymp} of the present lemma.

It then remains to prove part \ref{it:bn(psi)^2 asymp} of the lemma i.e. establish a two-term asymptotics for
$b_{n}(\psi)^2$. To achieve that we first evaluate $b_{n}(\psi)$. From the definition \eqref{eq:b def} of $b_{n}(\psi)$ we have,
using \eqref{eq:Pn' decay} to replace $E_{n}=n(n+1)$ by $n^2$ and $$\cos(\psi/m)=1+O\left(\frac{\psi^2}{n^2}\right)$$
to replace $\cos(\psi/m)$ by 1,
\begin{equation*}
\begin{split}
b_{n}(\psi) &= \frac{1}{n^2}P_{n}'(\cos(\psi/m))\cos(\psi/m) - \frac{1}{n^2}P_{n}''(\cos(\psi/m))\sin(\psi/m)^2 \\&+
\frac{1}{n^2}P_{n}(\cos(\psi/m))P_{n}'(\cos(\psi/m))^2\sin(\psi/m)^2+O\left( \frac{1}{ \psi^{5/2}}+\frac{1}{n\sqrt{\psi}} \right)
\\&= \frac{1}{n^2}P_{n}'(\cos(\psi/m)) - \frac{1}{n^2}P_{n}''(\cos(\psi/m))\sin(\psi/m)^2 \\&-
\frac{1}{n^2}P_{n}(\cos(\psi/m))P_{n}'(\cos(\psi/m))^2\sin(\psi/m)^2+O\left( \frac{1}{ \psi^{5/2}}  +\frac{1}{n\sqrt{\psi}}\right).
\end{split}
\end{equation*}
\fixme{changed sign in expression above}
Next we use the differential equation \eqref{eq:rec form Pn''} satisfied by the Legendre polynomials to write
\begin{equation*}
\begin{split}
b_{n}(\psi) &=
P_{n}(\cos(\psi/m))-\frac{1}{n^2}P_{n}'(\cos(\psi/m))\\&-
\frac{1}{n^2}P_{n}(\cos(\psi/m))P_{n}'(\cos(\psi/m))^2\sin(\psi/m)^2+O\left( \frac{1}{ \psi^{5/2}}  +\frac{1}{n\sqrt{\psi}}\right)
\\&= \sqrt{\frac{2}{\pi n \sin(\psi/m)}}\left(  \sin(\psi+\frac{\pi}{4})-\frac{1}{8}\frac{\cos(\psi+\frac{\pi}{4})}{\psi} \right)
-\sqrt{\frac{2}{\pi}}\frac{ \sin\left(\psi-\frac{\pi}{4}\right)}{n^{3/2}\sin(\psi/m)^{3/2}}
\\&-\sqrt{\frac{2}{\pi n \sin(\psi/m)}}  \sin(\psi+\frac{\pi}{4}) \cdot
\frac{1-\sin(2\psi)}{\pi n \sin(\psi/m)} +O\left( \frac{1}{ \psi^{5/2}}  +\frac{1}{n\sqrt{\psi}}\right),
\end{split}
\end{equation*}
\fixme{changed sign in expression above ($2$ places)}
where we used \eqref{eq:Pn asymp}, \eqref{eq:Pn' asymp} and reused \eqref{eq:1/n^2Pn'^2sin(psi/m)^2 asymp} once more
to obtain the second equality.
Reorganizing the terms in the last expression, we have
\begin{equation*}
\begin{split}
b_{n}(\psi) &=
\sqrt{\frac{2}{\pi n \sin(\psi/m)}}\sin(\psi+\frac{\pi}{4}) +
\frac{7}{8} \sqrt{\frac{2}{\pi n \sin(\psi/m)}}\frac{\cos(\psi+\frac{\pi}{4})}{\psi}
\\&-\frac{\sqrt{2} \cdot \left(\sin(\psi+\frac{\pi}{4})+\frac{1}{2}\cos(3\psi+\frac{\pi}{4})-\frac{1}{2}\cos(\psi-\frac{\pi}{4})   \right)}
{\pi^{3/2}n^{3/2}\sin(\psi/m)^{3/2}}
+O\left( \frac{1}{ \psi^{5/2}}  +\frac{1}{n\sqrt{\psi}}\right)
\\&=
\sqrt{\frac{2}{\pi n \sin(\psi/m)}}\sin(\psi+\frac{\pi}{4})
+\frac{7}{8} \sqrt{\frac{2}{\pi n \sin(\psi/m)}}\frac{\cos(\psi+\frac{\pi}{4})}{\psi}
\\&-\frac{\left(\sin(\psi+\frac{\pi}{4})+\cos(3\psi+\frac{\pi}{4})   \right)}{\sqrt{2}\pi^{3/2}n^{3/2}\sin(\psi/m)^{3/2}}
+O\left( \frac{1}{ \psi^{5/2}}  +\frac{1}{n\sqrt{\psi}}\right),
\end{split}
\end{equation*}
\fixme{changed sign in expression above ($2$ places)}
and we obtain part \ref{it:bn(psi)^2 asymp} of the present lemma by squaring the last equality.

\end{proof}

\section{Proof of Theorem \ref{thm:var smt lin stat}}
\label{sec:smt stat proof}

In this section we assume that $\varphi:\sphere\rightarrow\R$ is a {\em continuously differentiable even} function.
For the sake of proving Theorem \ref{thm:var gen lin stat}, we will conduct the analysis of the error terms
in terms of the $L^{\infty}$ norm $\|\varphi \|_{\infty}$ and the total variation $V(\varphi)$
of the test function, as prescribed by Theorem \ref{thm:var smt lin stat}.

Our first goal is to formulate an analogue of Proposition \ref{prop:VarZ expl scal} for the
variance of $\length^{\varphi}(f_{n})$. It turns out
that a certain auxiliary function $W^{\varphi}$ defined below comes out
from a straightforward repetition of the steps we performed in section \ref{sec:scal 2pnt corr},
adapted to suit $\length^{\varphi}$ rather than $\length$.

For $\varphi\in C^{1}(\sphere)$ the analogue of \eqref{eq:KacRice var} is
\begin{equation*}
\E\left[ \length(f_{n}) ^2\right] = \iint\limits_{\sphere\times\sphere} \varphi(x)\varphi(y)\tilde{K}_{n}(x,y) dx dy,
\end{equation*}
where $\tilde{K}_{n}(x,y)$ is given again by \eqref{eq:tildK def}. Since $\tilde{K}(x,y)=\tilde{K}(\phi)$,
where $\phi=d(x,y)$, we may employ Fubini, to obtain (cf. \eqref{eq:2nd mom spher coord})
\begin{equation*}
\E\left[ \length^{\varphi}(f_{n}) ^2\right] = 2\pi |\sphere |\int\limits_{0}^{\pi} \tilde{K}_{n}(\phi) W^{\varphi}(\phi) d\phi,
\end{equation*}
where $W^{\varphi}:[0,\pi]\rightarrow\R$ is a continuously differentiable function defined by
\begin{equation}
\label{eq:W transform def}
W^{\varphi}(\phi) := \frac{1}{8\pi^2}\int\limits_{d(x,y)=\phi}\varphi(x)\varphi(y) dxdy.
\end{equation}
For example, for the constant function $\varphi\equiv 1$ we have
\begin{equation*}
W^{1}(\phi) = \sin(\phi).
\end{equation*}
It is easy to check that, since $d(x,-y)=\pi-d(x,y)$, we have
\begin{equation}
\label{eq:Wphi(pi-phi)=Wphi(phi)}
W^{\varphi}(\pi-\phi) = W^{\varphi}(\phi),
\end{equation}
as we assume that $\varphi$ is even.

Scaling the integrand in the same manner exactly as in section \ref{sec:scal 2pnt corr}, we
finally obtain the following lemma (cf. Proposition \ref{prop:VarZ expl scal}).

\begin{lemma}
The variance of $\length^{\varphi}(f_{n})$ is given by

\begin{equation}
\label{eq:var(lenphi)=n*In}
\var(\length^{\varphi}(f_{n})) = 4\pi^{2} \frac{\eigval}{n+1/2} I_{n}^{\varphi},
\end{equation}
where
\begin{equation*}
I_{n}^{\varphi} = \int\limits_{0}^{\pi m}\left( K_{n}(\psi) - \frac{1}{4} \right) W^{\varphi}(\psi/m) d\psi
\end{equation*}
\end{lemma}

\begin{remark}
One deduces from \eqref{eq:Kn(psi)=Kn(pi m - psi)} and \eqref{eq:Wphi(pi-phi)=Wphi(phi)} that
\begin{equation}
\label{eq:Inphi 2int[0,pi m/2]}
I_{n}^{\varphi} =
2\int\limits_{0}^{\pi m/2}\left( K_{n}(\psi) - \frac{1}{4} \right) W^{\varphi}(\psi/m) d\psi.
\end{equation}

\end{remark}

We will need some rather simple properties of $W^{\varphi}$.
Writing the double integral \eqref{eq:W transform def} as an iterated integral and using the spherical coordinates with
pole at $x$ for each $x\in\sphere$ we obtain
\begin{equation}
\label{eq:W=sin(phi)*W0}
W^{\varphi}(\phi) = \frac{1}{8\pi^2}\sin(\phi) W^{\varphi}_{0}(\phi),
\end{equation}
where
\begin{equation*}
W^{\varphi}_{0}(\phi)= \int\limits_{\sphere}\varphi(x)dx\int\limits_{ST_{x}(\sphere)} \varphi(\exp_{x}(d\cdot \eta))   d\eta
\end{equation*}
is a continuously differentiable function with
\begin{equation}
\label{eq:W0(0)=2pi||varphi||^2}
W^{\varphi}_{0}(0) = 2\pi \| \varphi \|^2_{L^{2}(\sphere)},
\end{equation}
whose values are uniformly bounded by
\begin{equation}
\label{eq:W0<<M1^2}
|W^{\varphi}_{0}(\phi)| \le 2\pi \|\varphi \|_{\infty}\| \varphi \|_{L^{1}(\sphere)} \le 8\pi^2 \|\varphi \|_{\infty}^2
\end{equation}
and derivative uniformly bounded by
\begin{equation}
\label{eq:W0' << varphi' L1 norm}
|{W^{\varphi}_{0}}'(\phi)| \le 2\pi \|\varphi \|_{\infty} V(\varphi).
\end{equation}

Now we pursue the proof of Theorem \ref{thm:var smt lin stat}. By \eqref{eq:var(lenphi)=n*In},
evaluating the variance of $\length^{\varphi}$ is equivalent to evaluating $I_{n}^{\varphi}$,
and for notational convenience we choose to work with the expression \eqref{eq:Inphi 2int[0,pi m/2]}.
As in the proof of Theorem \ref{thm:var length}, we choose a constant $C>0$, which remains fixed throughout
the present section, and divide the interval $[0,\pi m /2] = [0,C]\cup [C,\pi m /2] $ (see section \ref{sec:In asympt}),

We then have the following lemma (cf. Lemma \ref{lem:In=2tildIn+O()}); to prove it just use \eqref{eq:W=sin(phi)*W0} and
the bound \eqref{eq:W0<<M1^2} for $W_{0}^{\varphi}$ together with Lemma \ref{lem:sing int small}.

\begin{lemma}
For any constant $C>0$, we have as $n\rightarrow\infty$
\begin{equation}
\label{eq:Inphi=2tildInphi+O(1/n)}
I_{n}^{\varphi} = 2\tilde{I}_{n}^{\varphi} + O_{\| \varphi \|_{\infty}}\left(\frac{1}{n}\right),
\end{equation}
where
\begin{equation}
\label{eq:tildInvarphi def}
\tilde{I}_{n}^{\varphi} := \int\limits_{C}^{\pi m /2}\left( K_{n}(\psi) - \frac{1}{4} \right)
W^{\varphi}\left(\frac{\psi}{m} \right) d\psi.
\end{equation}

\end{lemma}

\begin{proof}[Proof of Theorem \ref{thm:var smt lin stat}]

First we evaluate $\tilde{I}_{n}^{\varphi}$ as defined in \eqref{eq:tildInvarphi def}.
Plugging \eqref{eq:K(psi) asympt} into \eqref{eq:tildInvarphi def},
we have (cf. \eqref{eq:In=int(asympt)+error})
\begin{equation}
\label{eq:Invarphi=int(asympt)+error}
\begin{split}
\tilde{I}_{n}^{\varphi} &= \int\limits_{C}^{\pi m/2} \bigg[
\frac{1}{2}\frac{\sin(2\psi)}{\pi n \sin(\psi/m)} +  \frac{1}{256}\frac{1}{\pi^2 n \sin(\psi/m) \psi}
+\frac{9}{32}\frac{\cos(2\psi)}{\pi n \psi \sin(\psi/m)} \\&+
\frac{\frac{27}{64}\sin(2\psi)-\frac{75}{256}\cos(4\psi)}{\pi^2 n \psi \sin(\psi/m)}\bigg]
W^{\varphi}\left(\frac{\psi}{m}   \right)d\psi
+O\left( \int\limits_{C}^{\pi m /2}\left[\frac{1}{ \psi^{3}}
+\frac{1}{n\psi} \right]W^{\varphi}\left(\frac{\psi}{m}   \right)d\psi\right)
\\&= \frac{1}{16\pi^3 n}\int\limits_{C}^{\pi m/2} \bigg[
\sin(2\psi) +  \frac{1}{128}\frac{1}{\pi \psi}
+\frac{9}{16}\frac{\cos(2\psi)}{\psi} +
\frac{\frac{27}{32}\sin(2\psi)-\frac{75}{128}\cos(4\psi)}{\pi \psi}\bigg]
W_{0}^{\varphi}\left(\frac{\psi}{m}   \right)d\psi
\\&+O\left( \|\varphi \|_{\infty}^2\frac{1}{n}   \right),
\end{split}
\end{equation}
with\fixme{Changed $65$ into $1$ (2 places), $-11$ into $-75$
(2 places)} with constants involved in the $``O"$-notation universal.
Here we used the identity \eqref{eq:W=sin(phi)*W0}; to effectively control the error term we use \eqref{eq:W0<<M1^2}.

We integrate by parts the first oscillatory term in \eqref{eq:Invarphi=int(asympt)+error}, using
the continuous differentiability assumptions; this yields the bound for its contribution
\begin{equation*}
\begin{split}
&\ll \frac{1}{n} \int\limits_{C}^{\pi m/2} \sin(2\psi) W_{0}^{\varphi}(\psi/m) d\psi
\\&\ll
\frac{1}{n}  \left| \cos(2\psi) \cdot W_{0}^{\varphi}(\psi/m) \right| \big|_{\psi=\pi m/2}^{C}
+ \frac{1}{n^2}\left| \int\limits_{C}^{\pi m/2} \cos(2\psi) {W_{0}^{\varphi}}'(\psi/m)d\psi  \right|
\\&\ll \frac{\|\varphi \|_{\infty}^2}{n}+\frac{\|{W_{0}^{\varphi}}'\|_{L^{1}([0,\pi])}}{n} \ll (\|\varphi \|_{\infty}^2+
\|\varphi \|_{\infty}V(\varphi))\cdot \frac{1}{n}
\end{split}
\end{equation*}
with constants involved in the $``\ll"$-notation universal, by \eqref{eq:W0' << varphi' L1 norm}.
It is easy to establish similar bounds for the remaining oscillatory terms in \eqref{eq:Invarphi=int(asympt)+error}
i.e. the 3rd and the 4th terms.

To analyze the main contribution, which comes from the remaining second term in \eqref{eq:Invarphi=int(asympt)+error},
we note that the continuous differentiability of $W^{\varphi}_{0}$ implies
\begin{equation*}
W_{0}(\phi) = 2\pi \|\varphi\|^{2}_{L^{2}(\sphere)} + O_{\|\varphi \|_{\infty},V(\varphi)}(\phi),
\end{equation*}
by \eqref{eq:W0(0)=2pi||varphi||^2} and \eqref{eq:W0' << varphi' L1 norm}.
The main contribution to \eqref{eq:Invarphi=int(asympt)+error} is then
\begin{equation*}
\begin{split}
\frac{1}{2048\pi^4}\frac{1}{n} \int\limits_{C}^{\pi m/2}\frac{W_{0}^{\varphi}(\psi/m)}{\psi} d\psi
&= \frac{1}{1024\pi^3} \|\varphi \| _{L^{2}(\sphere)}^{2}
\cdot\frac{1}{n} \int\limits_{C}^{\pi m/2} \frac{d\psi}{\psi}
+ O_{\|\varphi \|_{\infty},V(\varphi)}\left( \frac{1}{n^2}\int\limits_{C}^{\pi m/2} d\psi \right)
\\&= \frac{1}{1024\pi^3} \|\varphi \| _{L^{2}(\sphere)}^{2} \cdot \frac{\log{n}}{n} +
O_{\|\varphi \|_{\infty},V(\varphi)}\left(\frac{1}{n}\right).
\end{split}
\end{equation*}
\fixme{Changed $65$ into $1$ (3 places)}

All in all we evaluated $\tilde{I}_{n}^{\varphi}$ as
\begin{equation*}
\tilde{I}_{n}^{\varphi}= \frac{1}{1024\pi^3} \|\varphi \| _{L^{2}(\sphere)}^{2} \cdot \frac{\log{n}}{n} + O_{\|\varphi \|_{\infty},
V(\varphi)}\left(\frac{1}{n}\right).
\end{equation*}
\fixme{Changed $65$ into $1$}
Plugging this into \eqref{eq:Inphi=2tildInphi+O(1/n)} yields
\begin{equation}
\label{eq:Inphi asympt}
I_{n}^{\varphi} = \frac{1}{512\pi^3} \|\varphi \| _{L^{2}(\sphere)}^{2} \cdot \frac{\log{n}}{n} + O_{\|\varphi \|_{\infty},
V(\varphi)}\left(\frac{1}{n}\right).
\end{equation}
\fixme{Changed $65$ into $1$}
We finally obtain the statement of Theorem \ref{thm:var smt lin stat} by plugging \eqref{eq:Inphi asympt} into \eqref{eq:var(lenphi)=n*In}.

\end{proof}

\section{Proof of Theorem \ref{thm:var gen lin stat}}
\label{sec:gen stat proof}

As implied by the formulation of Theorem \ref{thm:var gen lin stat}, in this section we will deal
with functions of bounded variation. The definition and some basic properties of the class $BV(\sphere)$
of functions of bounded variation is given in Appendix \ref{apx:BV functions}.

\subsection{On the proof of Theorem \ref{thm:var gen lin stat}}
\label{sec:on proof Thm BV}

To prove Theorem \ref{thm:var gen lin stat} one wishes to apply a standard approximation
argument, approximating our test function $\varphi$ of bounded variation with a sequence $\varphi_{i}$ of
$C^{\infty}$, for which we can apply Theorem \ref{thm:var smt lin stat}. There are two
major issues with this approach however.

On one hand, one needs to
check that $\varphi_{i}$ approximating $\varphi$ implies
the corresponding statement for the random variables $\length^{\varphi}(f_{n})$ and
$\length^{\varphi_{i}}(f_{n})$, and, in particular, their variance. While it is easy to check that if
$\varphi_{i}\rightarrow\varphi$ in $L^{1}$ then for every fixed $n$ we also have
$$\E[\length^{\varphi_{i}}(f_{n})]\rightarrow \E[\length^{\varphi}(f_{n})],$$ the analogous statement
for the variance is much less trivial (see Proposition \ref{prop:Elenphi^2<<n^2|phi|^2}\footnotemark).

\footnotetext{Proposition \ref{prop:Elenphi^2<<n^2|phi|^2} gives a stronger claim.
First, it evaluates the second moment rather than the variance. Secondly, it gives
a general bound for $E\left[\left(\Z^{\varphi_{i}}(f_{n})-\Z^{\varphi}(f_{n})\right)^2\right] =
E\left[\left(\Z^{\varphi_{i}-\varphi}(f_{n})\right)^2\right]$. It is easy to derive
the result we need employing the triangle inequality.}

On the other hand, when applying Theorem \ref{thm:var smt lin stat} for $\varphi_{i}$, one needs to control the error term in
\eqref{eq:var gen lin stat}, which may a priori depend on $\varphi_{i}$.
To resolve the latter we take advantage\footnotemark  of the fact that Theorem \ref{thm:var smt lin stat}
allows us to control the dependency of the error term in \eqref{eq:var gen lin stat} on the test function in terms
of its $L^{\infty}$ norm and total variation. Thus to resolve this issue it would be sufficient to require
from $\varphi_{i}$ to be essentially uniformly bounded and having uniformly bounded total variation.

\footnotetext{This is by no means a lucky coincidence; it is precisely the proof of Theorem \ref{thm:var gen lin stat}
that motivated the technical statement made in
Theorem \ref{thm:var smt lin stat}.}

Fortunately the standard symmetric mollifiers construction from ~\cite{GS} as given in Appendix
\ref{apx:BV functions} satisfy both the requirements above. Namely given a function $\varphi\in BV(\sphere)$
we obtain a sequence $\varphi_{i}$ of $C^{\infty}$ function,
that converge in $L^{1}$ to $\varphi$, $\|\varphi_{i} \|_{\infty} \le \|\varphi \|_{\infty}$ and
in addition $$V(\varphi_{i})\rightarrow V(\varphi).$$

\subsection{Continuity of the distribution of $\length^{\varphi}$}

As pointed in section \ref{sec:on proof Thm BV}, to prove Theorem \ref{thm:var gen lin stat} we will
need to show that the distribution of $\length^{\varphi}$ depends continuously on $\varphi$.
Proposition \ref{prop:Elenphi^2<<n^2|phi|^2} makes this statement precise. We believe that
it is of independent interest.

\begin{proposition}
\label{prop:Elenphi^2<<n^2|phi|^2}
Let $\varphi\in BV(\sphere)\cap L^{\infty}(\sphere)$ be any test function. Then
\begin{equation}
\label{eq:Elenphi^2<<n^2|phi|^2}
\E \left[\length^{\varphi}(f_{n})^2   \right] =
O\left( n^2 \|\varphi\|_{L^{1}(\sphere)}^{2}+ \|\varphi \|_{\infty} \|\varphi\|_{L^{1}(\sphere)} \right),
\end{equation}
where the constant involved in the $``O"$-notation are universal.
In particular, if $F\subseteq\sphere$ has a $C^{2}$ boundary then
\begin{equation*}
\E \left[\left(\length^F(f_{n})\right)^2   \right] = O(n^2 |F|^2+|F|).
\end{equation*}
\end{proposition}

\begin{proof}

Recall that we defined $W^{\varphi}$ as \eqref{eq:W transform def}; the assumption $\varphi\in L^{\infty}(\sphere)$
saves us from dealing with the validity of this definition. Starting from \eqref{eq:KacRice BV}, and
repeating the steps in proof of Lemma ~\ref{lem:KacRice gen} from either ~\cite{W1} or
~\cite{BSZ1,BSZ2}, we may extend the validity of the Kac-Rice formula
\eqref{eq:var(lenphi)=n*In} with \eqref{eq:Inphi 2int[0,pi m/2]} for this class
as well. Note that the constant term in \eqref{eq:Inphi 2int[0,pi m/2]} comes from the squared expectation,
so that we need to omit it if we want to compute the second moment.
We then have
\begin{equation}
\label{eq:EZphi^2 KacRice}
\E \left[\left(\length^{\varphi}(f_{n})\right)^2   \right] =  8\pi^2 \frac{\eigval}{n+1/2} J_{n}^{\varphi},
\end{equation}
where
\begin{equation*}
J_{n}^{\varphi} = \int\limits_{0}^{\pi m /2}K_{n}\left(\psi   \right) W^{\varphi}\left(\frac{\psi}{m}  \right)d\phi,
\end{equation*}
denoting as usual $m:=n+1/2$.

As usual while estimating this kind of integrals we remove the origin by choosing a constant $C>0$ and writing
\begin{equation}
\label{eq:Jn=Jn1+Jn2}
J_{n}^{\varphi} = J_{n,1}^{\varphi} + J_{n,2}^{\varphi},
\end{equation}
where
\begin{equation*}
J_{n,1}^{\varphi} = \int\limits_{0}^{C}K_{n}\left(\psi   \right)  W^{\varphi}\left(\frac{\psi}{m}\right) d\psi.
\end{equation*}
and
\begin{equation*}
J_{n,2}^{\varphi} = \int\limits_{C}^{\pi m /2}K_{n}\left(\psi   \right)  W^{\varphi}\left(\frac{\psi}{m}\right)d\psi.
\end{equation*}

First, for $C < \psi < \frac{\pi m}{2}$, $K_{n}(\psi)$ is bounded by a constant, which may depend only on $C$ i.e.
\begin{equation*}
|K_{n}(\psi)| = O_{C}(1),
\end{equation*}
which follows directly from Proposition \ref{prop:K asymp exp}. Therefore we may bound $J_{n,2}^{\varphi}$ as
\begin{equation}
\label{eq:Jn2<<n||phi||1^2}
\begin{split}
|J_{n,2}^{\varphi}| &\ll_{C} \int\limits_{C}^{\pi m /2}\left|W^{\varphi}\left(\frac{\psi}{m}\right)\right|d\psi \le
\int\limits_{0}^{\pi m /2}\left| W^{\varphi}\left(\frac{\psi}{m}\right)   \right| d\psi
\\&= m \int\limits_{0}^{\pi /2}\left| W^{\varphi}(\phi)   \right| d\psi \ll n \|\varphi\|_{L^{1}(\sphere)}^2,
\end{split}
\end{equation}
as earlier.

We claim that for $0 < \psi < C$ we may bound $K_{n}$ as
\begin{equation}
\label{eq:Kpsi << 1/psi}
|K_{n}(\psi)| =  O_{C}\left( \frac{1}{\psi}\right).
\end{equation}
Before proving this estimate we will show how it helps us to bound
$J_{n,1}^{\varphi}$. We have by the definition of $J_{n,1}^{\varphi}$
\begin{equation}
\label{eq:Jn1 << ||phi||inf ||phi||1}
\begin{split}
\left| J_{n,1}^{\varphi}\right| &\ll \int\limits_{0}^{C} \frac{1}{\psi} \left| W^{\varphi}\left(\frac{\psi}{m}\right)\right|d\psi
\ll \frac{1}{n}\int\limits_{0}^{C} \left| W_{0}^{\varphi}\left(\frac{\psi}{m}\right)\right|d\psi \\&\ll
\int\limits_{0}^{C/n} \left| W_{0}^{\varphi}\left(\phi\right) \right| d\phi
\ll_{C} \frac{1}{n} \|\varphi\|_{\infty} \|\varphi\|_{L^{1}(\sphere)},
\end{split}
\end{equation}
by \eqref{eq:W=sin(phi)*W0} and the first inequality of \eqref{eq:W0<<M1^2}.

The statement of the present lemma now follows from plugging the estimates \eqref{eq:Jn2<<n||phi||1^2}
and \eqref{eq:Jn1 << ||phi||inf ||phi||1} into \eqref{eq:Jn=Jn1+Jn2} and \eqref{eq:EZphi^2 KacRice}.
We still have to prove \eqref{eq:Kpsi << 1/psi} though.

To see \eqref{eq:Kpsi << 1/psi} we use Remark \ref{rem:K(psi) prob} and the Cauchy-Schwartz inequality
to write
\begin{equation}
\label{eq:Kn(psi) prob}
K_{n}(\psi) = \frac{1}{(2\pi)\sqrt{1-P_{n}(\cos\psi/m)^2}}\E\left[\| U \| \cdot \| V \|       \right],
\end{equation}
where $U$ and $V$ are $2$-dimensional mean zero Gaussian vectors with covariance matrix \eqref{eq:Delta def},
whose entries uniformly bounded by an absolute constant, whence
\begin{equation}
\label{eq:E[|U||V|]<<1}
\E\left[\| U \| \cdot \| V \| \right] \le \sqrt{\E\left[\| U \|^2\right] \E\left[ \| V \|^2 \right]} = O(1),
\end{equation}
with the constant involved in the ``$O$"-notation uniform. For the other term Lemma \ref{lem:1-Pn^2>>n^2phi^2}
yields
\begin{equation}
\label{eq:1-Pn(cos(psi/m))^2>>1/psi}
\sqrt{1-P_{n}(\cos(\psi/m))^2} \ll \frac{1}{\psi},
\end{equation}
so that we obtain the necessary bound \eqref{eq:Kpsi << 1/psi} for $K_{n}(\psi)$
plugging the estimates \eqref{eq:E[|U||V|]<<1} and \eqref{eq:1-Pn(cos(psi/m))^2>>1/psi} into
\eqref{eq:Kn(psi) prob}.

\end{proof}

\subsection{Proof of Theorem \ref{thm:var gen lin stat}}

Now we are ready to give a proof of Theorem \ref{thm:var gen lin stat}.

\begin{proof}[Proof of Theorem \ref{thm:var gen lin stat}]

Given a function $\varphi\in BV(\sphere)$, let $\varphi_{i}\in C^{\infty}$ be a sequence of smooth functions such that
$\varphi_{i} \rightarrow \varphi$ in $L^{1}(\sphere)$, $$V_{i}:=V(\varphi_{i}) \rightarrow V(\varphi), $$
and
\begin{equation}
\label{eq:phii <= M1}
\|\varphi_{i}\|_{\infty} \le \|\varphi \|_{\infty}.
\end{equation}
(see Appendix \ref{apx:BV functions}).
Let $M_{1} := \|\varphi \|_{\infty}$ and $$M_{2} := \max\limits \{V_{i}\}_{i\ge 1}  < \infty,$$ since $V_{i}$ is convergent.

Theorem \ref{thm:var smt lin stat} applied on $\varphi_{i}\in C^{\infty}(\sphere)$ states that
\begin{equation}
\label{eq:var(phi_i)=clogn+O(1)}
\var(\length^{\varphi_{i}}(f_{n})) = c(\varphi_{i})\cdot  \log{n} + O_{M_{1},M_{2}}(1),
\end{equation}
where $c(\varphi_{i})$ is given by
\begin{equation*}
c(\varphi_{i}):=\frac{\|\varphi_{i} \|_{L^{2}(\sphere)}^2 }{128 \pi} > 0.
\end{equation*}
\fixme{Changed $65$ into $1$}
Note that since $\varphi_{i}$ and $\varphi$ are uniformly bounded \eqref{eq:phii <= M1}, $L^1(\sphere)$ convergence implies
$L^{2}(\sphere)$ convergence, so that
\begin{equation}
\label{eq:c(phii)->c(phi)}
c(\varphi_{i})\rightarrow c(\varphi),
\end{equation}
the latter being given by \eqref{eq:c(phi) def}.

On the other hand we know from Proposition \ref{prop:Elenphi^2<<n^2|phi|^2} that
\begin{equation*}
\E\left[ \left( \length^{\varphi_{i}}(f_{n}) - \length^{\varphi}(f_{n})\right)^2   \right] =
\E\left[ \left( \length^{\varphi_{i}-\varphi}(f_{n})\right)^2   \right]\rightarrow 0,
\end{equation*}
using the uniform boundedness \eqref{eq:phii <= M1} again to ensure that \eqref{eq:Elenphi^2<<n^2|phi|^2}
holds uniformly. This together with the triangle inequality implies that
\begin{equation}
\label{eq:var(phi_i)->var(phi)}
\var\left(\length^{\varphi_{i}}(f_{n})\right) \rightarrow \var\left(  \length^{\varphi}(f_{n})   \right),
\end{equation}
and we take the limit $i \rightarrow\infty$ in
\eqref{eq:var(phi_i)=clogn+O(1)} to finally obtain the main statement of Theorem \ref{thm:var gen lin stat}.

\end{proof}

\begin{remark}
From the proof presented, it is easy to see that the constant in the $``O"$-notation in the statement \eqref{eq:var gen lin stat}
of Theorem \ref{thm:var gen lin stat}
could be made dependent only on $\|\varphi \|_{\infty}$ and $V(\varphi)$.
\end{remark}

\appendix

\section{Computation of the covariance matrix}
\label{apx:covar mat}

In this section we compute the matrix $\Omega_{n}(\phi)$ explicitly, as prescribed by
\eqref{eq:Omega(phi) def}.
The matrix $\Omega_{n}(\phi)$ is the $4\times 4$ covariance matrix of the mean zero Gaussian random vector $Z_{2}$
in \eqref{eq:Z=f(x)f(y)nabla} with $x\ne y\in\sphere$ any
two points on the arc $\{ \theta=0 \}$ with $d(x,y)=\phi$, conditioned upon $f(x)=f(y)=0$.
Recall that as such, $\Omega_{n}(\phi)$ is given by \eqref{eq:Omega def gen frames},
where $A=A_{n}(x,y)$, $B=B_{n}(x,y)$ and $C=C_{n}(x,y)$
are given by \eqref{eq:A blk def}, \eqref{eq:B blk def} and \eqref{eq:C blk def} respectively,
and $x,y\in\sphere$ are any points on the arc $\{ \theta=0 \}$ with $d(x,y)=\phi$.
Here the gradients are given in the orthonormal frame \eqref{eq:orthnorm frame spher}
of the tangent planes $T_{x}(\sphere)$ associated to the spherical coordinates
(see section \ref{sec:spher orth frames} for explanation).

Let $x$ and $y$ correspond to the spherical coordinates $(\phi_{x},\theta_{x}=0)$ and
$(\phi_{y},\theta_{y}=0)$, and denote $$\phi=d(x,y) = |\phi_{x}-\phi_{y}|.$$
Recall that
\begin{equation*}
u_{n}(x,y) = P_{n}(\cos(d(x,y))) = P_{n}(\cos\phi).
\end{equation*}

First we compute the inverse of $A$ in \eqref{eq:A blk def} as
\begin{equation}
\label{eq:A^-1 comp}
A_{n}(\phi)^{-1} = \frac{1}{1-P_{n}(\cos\phi)^2} \left( \begin{matrix}
  1    &-P_{n}(\cos\phi) \\ -P_{n}(\cos\phi) &1
\end{matrix}           \right)
\end{equation}

It is easy to either see from the geometric picture or compute explicitly that
\begin{equation}
\label{eq:gradxy u comp}
\nabla_{x} u_{n}(x,y) = -\nabla_{y} u_{n}(x,y) = \pm P'_{n}(\cos\phi)\sin(\phi) (1,0),
\end{equation}
depending on whether $\phi_{x} > \phi_{y}$ or $\phi_{x} < \phi_{y}$, so that
\begin{equation}
\label{eq:B comp}
B_{n}(\phi) = \pm \left(   \begin{matrix}
  0 &0 &P'_{n}(\cos\phi)\sin\phi &0 \\ -P'_{n}(\cos\phi)\sin\phi &0 &0 &0
\end{matrix}   \right).
\end{equation}

Next we turn to the missing part of $C_{n}(\phi)$ defined in \eqref{eq:C blk def}, i.e. the
``pseudo-Hessian'' $H_{n}(\phi)$ given by \eqref{eq:pseudo-Hessian of u}.
By the chain rule
\begin{equation}
\label{eq:H chain rule}
\begin{split}
H_{n}(\phi) &= \left(\nabla_{x} \otimes \nabla_{y} \right) u_{n}(x,y) =
\nabla_{x} \otimes P_{n}'(\cos(d(x,y)))\nabla_{y}\cos(d(x,y))  \\&=
P_{n}''(\cos\phi)\nabla_{x} \cos(d(x,y)) \otimes
\nabla_{y} \cos(d(x,y))\\&+ P_{n}'(\cos\phi)\left( \nabla_{x} \otimes \nabla_{y} \right) \cos(d(x,y)).
\end{split}
\end{equation}

We denote
\begin{equation*}
h(x,y) := \cos d(x,y) = \cos\phi_{x}\cos\phi_{y}+\sin\phi_{x}\sin\phi_{y}\cos(\theta_{x}-\theta_{y}),
\end{equation*}
and compute explicitly that for $\theta_{x}=\theta_{y}=0$ we have
\begin{equation}
\label{eq:Hessian cosd comp}
\left(\nabla_{x}\otimes\nabla_{y}\right) \cos(d(x,y))
=\left(\nabla_{x}\otimes\nabla_{y}\right) h(x,y) = \left( \begin{matrix} \cos\phi   &0 \\ 0 &1  \end{matrix}     \right).
\end{equation}

Plugging \eqref{eq:gradxy u comp} and \eqref{eq:Hessian cosd comp} into \eqref{eq:H chain rule} we obtain
\begin{equation}
\label{eq:H comp}
H = \left(  \begin{matrix}
P'_{n}(\cos\phi)\cos\phi-P''_{n}(\cos\phi)\sin(\phi)^2 &0\\0 & P_{n}'(\cos\phi).
\end{matrix}    \right)
\end{equation}
Finally plugging \eqref{eq:H comp} into \eqref{eq:C blk def}, and plugging that together with \eqref{eq:A^-1 comp} and
\eqref{eq:B comp} into \eqref{eq:Omega def gen frames}, we obtain an explicit expression for $\Omega_{n}(\phi)$
as prescribed by \eqref{eq:Omega(phi) def} with entries given by \eqref{eq:tild a def}, \eqref{eq:tild b def} and
\eqref{eq:tild c def}.

\section{Estimates for the Legendre polynomials and related functions}
\label{apx:legendre pol}

The goal of this section is to give a brief introduction to the Legendre polynomials
$P_{n}:[-1,1]\rightarrow\R$ and give some relevant basic information necessary for the purposes
of the present paper. The high degree asymptotic analysis of behaviour of $P_{n}$
and its first two derivatives involves the Hilb's asymptotics in Lemma
\ref{lem:Hilb asymp gen} together with the recursion \eqref{eq:rec form Pn'} for the 1st derivative and
the differential equation \eqref{eq:rec form Pn''} for the second one.
We refer the reader to ~\cite{S} for more information.

The Legendre polynomials $P_{n}$ are defined as the unique polynomials
of degree $n$ orthogonal w.r.t. the constant weight function $\omega(t)\equiv 1$
on $[-1,1]$ with the normalization $P_{n}(1)=1$. They satisfy the following second order differential equation:
\begin{equation}
\label{eq:rec form Pn''}
\begin{split}
P_{n}''(\cos(\psi/m)) =
-\frac{n(n+1)}{\sin(\psi/m)^2} P_{n}(\cos(\psi/m))+\frac{2\cos(\psi/m)}{\sin(\psi/m)^2}P_{n}'(\cos(\psi/m)),
\end{split}
\end{equation}
as well as the recursion
\begin{equation}
\label{eq:rec form Pn'}
P_{n}'(\cos(\psi/m)) = \left(P_{n-1}\left(\cos(\psi/m)\right)-\cos(\psi/m)P_{n}\left(\cos(\psi/m)\right)\right)\frac{n}{\sin(\psi/m)^2}.
\end{equation}

The Hilb asymptotics gives the high degree asymptotic behaviour of $P_{n}$.
\begin{lemma}[Hilb Asymptotics (formula (8.21.17) on page 197 of Szego ~\cite{S})]
\label{lem:Hilb asymp gen}
\begin{equation}
\label{eq:Hilb asymp gen}
P_{n}(\cos{\phi}) =
\bigg(\frac{\phi}{\sin{\phi}}\bigg)^{1/2}J_{0}((n+1/2)\phi)+\delta(\phi),
\end{equation}
uniformly for $0\le\phi\le\pi/2$, $J_{0}$ is the Bessel $J$ function of order $0$ and the
error term is
\begin{equation*}
\begin{split} \delta(\phi) \ll \begin{cases}
\phi^{1/2} O(n^{-3/2}), \: &Cn^{-1} < \phi < \pi/2 \\
\phi^{\alpha+2}O(n^{\alpha}), \: &0<\phi < Cn^{-1},
\end{cases}
\end{split}
\end{equation*}
where $C>0$ is any constant and the constants involved in the $``O"$-notation depend on $C$ only.
\end{lemma}

We have the following rough estimate for the behaviour of the Legendre polynomials at $\pm 1$, which follows directly from
Hilb's asymptotic.
\begin{lemma}
\label{lem:1-Pn^2>>n^2phi^2}
For $0 < \phi < \frac{\pi}{2}$ one has
\begin{equation*}
1-P_{n}(\cos(\phi))^2 \gg n^2 \phi^2,
\end{equation*}
where the constant in the ``$\gg$"-notation is universal.
\end{lemma}

\begin{lemma}
\label{lem:Pn,Pn',Pn'' asymp est}

The Legendre polynomials $P_{n}$ and its couple of derivatives satisfy uniformly for $n \ge 1$, $\psi > C$:

\begin{enumerate}

\item
\begin{equation}
\label{eq:Pn asymp}
\begin{split}
P_{n}(\cos(\psi/m)) &=\sqrt{\frac{2}{\pi n \sin(\psi/m)}}\left(  \sin(\psi+\frac{\pi}{4})-\frac{1}{8}\frac{\cos(\psi+\frac{\pi}{4})}{\psi} \right)
\\&+ O\left(\frac{1}{\psi^{5/2}} + \frac{1}{\sqrt{\psi}n}\right)
\end{split}
\end{equation}

\item
\begin{equation}
\label{eq:Pn' asymp}
\begin{split}
&P_{n}'(\cos(\psi/m))=\\&\sqrt{\frac{2}{\pi}}\frac{\sqrt{n}}{\sin(\psi/m)^{5/2}} \left(
\sin(\psi/m)\sin\left( \psi-\frac{\pi}{4}  \right) + \frac{3}{8n}\sin\left( \psi+\frac{\pi}{4}   \right)\right) \\&+
O\left(\frac{n^2}{\psi^{7/2}}+  \frac{n}{\psi^{3/2}}  \right)
\end{split}
\end{equation}

\item
\begin{equation}
\label{eq:Pn'' simple rec}
P_{n}''(\cos(\psi/m))= -\frac{n^2}{\sin(\psi/m)^2}P_{n}(\cos(\psi/m)) + \frac{2}{\sin(\psi/m)^2}P_{n}'(\cos(\psi/m))
+O\left( \frac{n^3}{\psi^{5/2}}\right)
\end{equation}

\end{enumerate}

\end{lemma}

\begin{proof}

By Lemma \ref{lem:Hilb asymp gen} and the standard asymptotics for the Bessel functions we obtain
\begin{equation*}
\begin{split}
P_{n}(\cos(\psi/m)) &= \frac{\sqrt{\psi/m}}{\sqrt{\sin(\psi/m)}}J_{0}(\psi) + O(\frac{\sqrt{\psi}}{n^2})
\\&= \sqrt{\frac{2}{\pi}}\frac{\sqrt{\psi/m}}{\sqrt{\sin(\psi/m)}}\left( \frac{\sin(\psi+\frac{\pi}{4})}{\sqrt{\psi}}
-\frac{1}{8}\frac{\cos(\psi+\frac{\pi}{4})}{\psi^{3/2}}   \right) + O\left(\frac{1}{\psi^{5/2}}
+ \frac{\sqrt{\psi}}{n^2}\right)
\\&=\sqrt{\frac{2}{\pi n \sin(\psi/m)}}\left(  \sin(\psi+\frac{\pi}{4})-\frac{1}{8}\frac{\cos(\psi+\frac{\pi}{4})}{\psi} \right)
+ O\left(\frac{1}{\psi^{5/2}} + \frac{1}{\sqrt{\psi}n}\right),
\end{split}
\end{equation*}
which is \eqref{eq:Pn asymp}.

To obtain \eqref{eq:Pn' asymp} we employ the recursive formula \eqref{eq:rec form Pn'}, evaluating
the Legendre polynomials appearing there using \eqref{eq:Pn asymp}. Finally we obtain
a simple approximate differential equation \eqref{eq:Pn'' simple rec}, replacing
$n(n+1)$ by $n^2$ and $\cos(\psi/m)$ by $1$ in the differential equation
\eqref{eq:rec form Pn''} satisfied by the Legendre polynomials. To do so we use the decay
$$|P_{n}(\cos(\psi/m))| = O\left( \frac{1}{\sqrt{\psi}}  \right)$$ of $P_{n}$, which follows directly
from \eqref{eq:Pn asymp}, as well as \eqref{eq:Pn' decay} of its derivative.

\end{proof}

\section{Functions of bounded variation}
\label{apx:BV functions}

In this section we give the definition and some basic properties on the functions of bounded variation.
For more information we refer the reader to ~\cite{GS}.

Classically, the {\em variation} of a function $\eta:[a,b]\rightarrow\R$ on $[a,x]$ is defined as
\begin{equation*}
V(\eta; x) := \sup\limits_{\lambda:\: t_{1}=a<t_{2}<\ldots< t_{k}=x} \sum\limits_{i=1}^{k-1}|\eta(t_{i+1})-\eta(t_{i})|
\end{equation*}
where the supremum is over all the partitions $\lambda$ of $[a,x]$.
We denote $I:=[a,b]$. If $\eta\in C^{1}(I)$ then the variation is
\begin{equation*}
V(\eta;x) = \int\limits_{a}^{x} |\eta'(t)| dt.
\end{equation*}
In fact, the last inequality holds even for $\eta\in W^{1,1}(I)$, where for this class of functions the derivative
$\eta'$ is the {\em weak derivative}.

This definition has two major disadvantages. First, one wishes to identify functions
\begin{equation}
\label{eq:equiv rel fnc}
\eta_{1} \sim \eta_{2}, \text{ if }\eta_{1}(x) = \eta_{2}(x) \text{ for almost all }x\in I.
\end{equation}
However, altering the values of $\eta$ on a measure zero set does impact its variation.
Secondly, one cannot extend this definition for the multivariate case.

We then need to find a better definition. Fortunately, the following definition eliminates the disadvantages
of the previous one.
Let
\begin{equation*}
V(\eta;x) := \sup\limits_{g} \int\limits_{0}^{x} \eta(t)g'(t) dt,
\end{equation*}
where the supremum is over all the continuously differentiable functions $g:[a,x]\rightarrow\R$ with $|g(t)|\le 1$
for all $t\in [a,x]$.
The number $V(\eta):=V(\eta;I)$ is called the {\em total variation} of $\eta$ on $I$.
We define the space $BV(I)$ to be the equivalence classes of functions $\eta$ with finite total variation,
i.e.
\begin{equation*}
BV(I) := \{ \eta\in L^{1}(I):\:  V(\eta)<\infty \} /\sim,
\end{equation*}
where the equivalence relation is given by \eqref{eq:equiv rel fnc}. It is known ~\cite{GS} that
$$W^{1,1}(I) \subsetneq BV(I).$$

We may extend the latter definition quite naturally for the multivariate case. Of our interest is the case of
the sphere.
Let $\varphi\in L^{1}(\sphere)$ be an integrable function. We define its variation on an open subset
$\Omega \subseteq\sphere$ as
\begin{equation*}
V(\varphi;\Omega) := \sup\limits_{g} \int\limits_{\Omega} \varphi(x) \divg g(x)dx,
\end{equation*}
where the supremum is over the continuously differentiable compactly supported vector fields
$$g\in C^{1}_{c}(\Omega, T\Omega)$$ with $|g(x)|\le 1$ for all $x\in\Omega$.
We define the {\em total variation} as
\begin{equation*}
V(\varphi) := V(\varphi; \sphere).
\end{equation*}

The space $BV(\Omega)$ is defined as the equivalence class of functions $\varphi$ with $V(\varphi) < \infty$,
with the equivalence relation \eqref{eq:equiv rel fnc} adapted to the sphere.
Again, for a smooth (and $W^{1,1}(\sphere)$) function $\varphi\in C^{1}(\sphere)$ we have
\begin{equation*}
V(\varphi) = \int\limits_{\sphere} \| \nabla \varphi(x) \| dx,
\end{equation*}
and $$W^{1,1}(\Omega)\subsetneq BV(\Omega).$$

For a function $\varphi\in BV(\sphere)$ ~\cite{GS}, Theorem 1.17 gives a construction\footnotemark
of a sequence $\varphi_{i}\in C^{\infty}$ of smooth test functions such that
$\varphi_{i} \rightarrow \varphi$ in $L^{1}(\sphere)$ as well as
$$V(\varphi_{i}) \rightarrow V(\varphi).$$
Moreover, part (b) of that theorem implies that
\begin{equation*}
\|\varphi_{i}\|_{\infty} \le \|\varphi \|_{\infty}.
\end{equation*}

\footnotetext{This book gives only the theory of functions of bounded variation on $\R^{m}$.
One can obtain a similar theory for the sphere only slightly modifying the one given.}

We are interested in the linear statistics of the nodal sets of smooth functions, where the test functions are of bounded variation.
The definition of the linear statistics is natural for continuous test functions $\varphi:\sphere\rightarrow \R$ as
\begin{equation*}
\length^{\varphi}(f) = \int\limits_{f^{-1}(0)}\varphi(x) dx,
\end{equation*}
i.e. integrating the restriction of $\varphi$ on the nodal line.
However, things become more complicated as one drops the continuity assumption; since the values of $\varphi\in BV(\sphere)$
(or $\varphi\in L^{1}(\sphere)$) are only defined up to measure zero sets, there is no meaning to restricting
$\varphi$ on curves. In general, one cannot define linear statistics corresponding to integrable functions,
and to define a notion of {\em trace} of $\varphi$ on a smooth curve $C$,
we will have to exploit the values of $\varphi$ in a tubular neighbourhood around $C$. Such a construction is known
for the functions belonging to the class $W^{1,1}(\sphere)$, i.e. for every smooth curve $C\subseteq \Omega$ there exists a map
$$\tr_{C}:W^{1,1}(\Omega)\rightarrow L^{1}(C)$$ satisfying the natural properties.

The situation is more involved in the $BV$-case, which is essential to us, since $W^{1,1}$ does not contain the
characteristic functions of nice spherical subsets. A smooth curve divides the sphere and a tubular neighbourhood
around it into two parts. One may then define \cite{GS}, chapter 2, {\em two}
traces $\varphi^{+}=\tr_{C}^{+}\varphi$ and $\varphi^{-}=\tr_{C}^{-}\varphi$ both belonging to $L^{1}(C)$, corresponding to the
values of $\varphi$ on the different parts. The traces $\varphi^{+}$ and $\varphi^{-}$ may in general be different\footnotemark,
and moreover, one cannot canonically distinguish between the traces.
For instance, if $F\subseteq\sphere$ is a nice subset, and $\chi_{F}$ is its
characteristic function, then $\tr_{\partial F}(\chi_{F})$ might be defined as either $1$ or $0$, depending on whether we approach
the circle from inside or outside the disc respectively. Accordingly, the corresponding linear statistic
might be $\len(\partial F)$ or $0$.

\footnotetext{Intuitively, the traces $\varphi^{+}$ and $\varphi^{-}$ will be different precisely if the jump of $\varphi$
occurs on a subset of $C$, as follows from ~\cite{GS}, Proposition 2.8.
It is plausible that with probability $1$ this situation will not happen for the nodal lines of
spherical harmonics; we believe that this is a minor issue and of little interest to the present paper.
This situation is almost surely impossible for the characteristic functions of nice sets,
which are the main motivation for considering the class $BV$.}

We define the {\em average} trace of $\varphi$ on a smooth curve $C\subseteq\sphere$ as
\begin{equation}
\label{eq:phi+-def}
\varphi^{\pm} := \frac{1}{2}\varphi^{+}+\frac{1}{2}\varphi^{-},
\end{equation}
and this is the notion that appears in the formulation of Theorem \ref{thm:var gen lin stat} and
throughout the present paper. For $\varphi\in L^{\infty}(\sphere)$ we have
\begin{equation*}
\|\varphi^{\pm}\|_{\infty} \le \|\varphi\|_{\infty}.
\end{equation*}

Following the approach of ~\cite{GS}, (2.10) and Federer's co-area formula ~\cite{FD},
one may obtain the inequality
\begin{equation*}
\begin{split}
&\frac{1}{\epsilon}\int\limits_{0}^{\epsilon} \left| \int\limits_{f^{-1}(t)}\varphi(x)dx -
\int\limits_{f^{-1}(0)}\varphi^{+}(x)dx  \right|dt
\\&= O_{f}\left(V\left(\varphi; f^{-1}((0,\epsilon))\right)    + \sup\limits_{0<t <\epsilon} \left|\len(f^{-1}(t))-\len(f^{-1}(0)) \right|\right).
\end{split}
\end{equation*}
As $\beta\rightarrow 0$, the right hand side of the last inequality vanishes. Therefore
we have the following Kac-Rice type formula
\begin{equation*}
\int\limits_{f^{-1}(0)}\varphi^{+}(x)dx = \lim\limits_{\epsilon\rightarrow 0}
\frac{1}{\epsilon}\int\limits_{0 < f(x) < \epsilon}\|\nabla f(x)\| \varphi(x) dx,
\end{equation*}
and similarly
\begin{equation*}
\int\limits_{f^{-1}(0)}\varphi^{-}(x)dx = \lim\limits_{\epsilon\rightarrow 0}
\frac{1}{\epsilon}\int\limits_{-\epsilon < f(x) < 0}\|\nabla f(x)\| \varphi(x) dx.
\end{equation*}
Combining the last two formulas we obtain
\begin{equation}
\label{eq:KacRice BV}
\int\limits_{f^{-1}(0)}\varphi^{\pm}(x)dx = \lim\limits_{\epsilon\rightarrow 0}
\frac{1}{2\epsilon}\int\limits_{|f(x)| < \epsilon} \|\nabla f(x)\| \varphi(x) dx.
\end{equation}
We employ \eqref{eq:KacRice BV} to extend the validity of the Kac-Rice formula
for the second moment for $\varphi\in BV(\sphere)$ (see \eqref{eq:EZphi^2 KacRice}).


\begin{thebibliography}{99}

\bibitem{AAR}
Andrews, George E.; Askey, Richard; Roy, Ranjan {\em Special functions
Encyclopedia of Mathematics and its Applications} 71. Cambridge
University Press, Cambridge, 1999.


\bibitem{Berard}
B\'erard, P. {\em Volume des ensembles nodaux des fonctions propres
du laplacien}.   Bony-Sjostrand-Meyer seminar, 1984--1985, Exp. No.
14 , 10 pp., \'Ecole Polytech., Palaiseau, 1985.


\bibitem{Berry 2002}
Berry, Michael V.  {\em Statistics of nodal lines and points in chaotic
quantum billiards: perimeter corrections, fluctuations, curvature}
J. Phys. A {\bf 35} (2002), 3025-3038.

\bibitem{BSZ1}
Bleher, Pavel; Shiffman, Bernard; Zelditch, Steve {\em Universality and scaling of correlations between zeros on complex manifolds} Invent. Math. 142 (2000), no. 2, 351--395.

\bibitem{BSZ2}
Bleher, Pavel; Shiffman, Bernard; Zelditch, {\em Steve Universality and scaling of zeros on symplectic manifolds}
Random matrix models and their applications, 31--69, Math. Sci. Res. Inst. Publ., 40, Cambridge Univ. Press, Cambridge, 2001.

\bibitem{Bruning}
J. Br\"{u}ning {\em \"{U}ber Knoten Eigenfunktionen des
Laplace-Beltrami Operators}, Math. Z. 158 (1978), 15--21.


\bibitem{Bruning-Gromes}
J. Br\"{u}ning and D. Gromes {\em \"{U}ber die L\"{a}nge der
Knotenlinien schwingender Membranen}, Math. Z. 124 (1972), 79--82.

\bibitem{Cheng}
S.~Y.~Cheng,  {\em Eigenfunctions and nodal sets}, Comm. Math. Helv.
{\bf 51} (1976), 43--55.


\bibitem{CL}
Cram\'{e}r, Harald; Leadbetter, M. R. Stationary and related stochastic processes.
Sample function properties and their applications. Reprint of the 1967 original.
Dover Publications, Inc., Mineola, NY, 2004.

\bibitem{Donnelly-Fefferman}
H.~Donnelly, and C.~Fefferman {\em Nodal sets of eigenfunctions on
Riemannian manifolds}, Invent. Math. {\bf 93} (1988), 161--183.


\bibitem{FD}
Federer, Herbert; {\em Curvature measures} Trans. Amer. Math. Soc. 93 1959 418--491.

\bibitem{FH}
Forrester, P. J.; Honner, G. Exact statistical properties of the zeros of complex random polynomials.  J. Phys. A  32  (1999),  no. 16, 2961--2981.

\bibitem{GS}
Giusti, Enrico {\em Minimal surfaces and functions of bounded variation}
Monographs in Mathematics, 80. Birkh�user Verlag, Basel, 1984.


\bibitem{KW}
Krishnapur, Manjunath; Wigman, Igor {\em Fluctuations of the nodal length of random eigenfunctions of the Laplacian on the torus},
in preparation.

\bibitem{Neuheisel}
J. Neuheisel, {\em The asymptotic distribution of nodal sets on
  spheres}, Johns Hopkins Ph.D. thesis (2000).

\bibitem{RW}
 Z.~Rudnick and I.~Wigman
{\em On the volume of nodal sets for eigenfunctions of the Laplacian
on the torus}, Annales Henri Poincare, Vol. 9 (2008), No. 1, 109--130


\bibitem{S}
Szego, Gabor {\em Orthogonal polynomials.} Fourth edition. American
Mathematical Society, Colloquium Publications, Vol. XXIII. American
Mathematical Society, Providence, R.I., 1975.

\bibitem{SZ1}
Shiffman, Bernard; Zelditch, Steve {\em Number variance of random zeros on complex manifolds. Geom. Funct. Anal. }
18 (2008), no. 4, 1422--1475.

\bibitem{SZ2} Shiffman, Bernard; Zelditch, Steve
{\em Number variance of random zeros on complex manifolds, II: smooth statistics}, available online
http://arxiv.org/abs/0711.1840

\bibitem{ST}
Sodin, Mikhail; Tsirelson, Boris {\em Random complex zeroes. I. Asymptotic normality} Israel J. Math. 144 (2004), 125--149.

\bibitem{TW1}
Toth, John A.; Wigman, Igor {\em Counting open nodal lines of random waves on planar domains},
IMRN (2009).

\bibitem{TW2}
Toth, John A.; Wigman, Igor {\em Universality of length distribution of nodal lines of random waves on generic surfaces},
in progress (2009)


\bibitem{W1}
Wigman, I. {\em On the distribution of the nodal sets of random spherical harmonics}
Jour. Math. Phys (2009)

\bibitem{W2}
Wigman, I. {\em Volume fluctuations of the nodal sets of random Gaussian subordinated spherical harmonics},
in preparation.

\bibitem{Y1} Yau, S.T. Survey on partial differential equations in differential geometry.
{\it Seminar on Differential Geometry}, pp. 3--71, Ann. of Math. Stud., 102, Princeton Univ. Press, Princeton, N.J., 1982.


\bibitem{Y2} Yau, S.T. Open problems in geometry.
{\it  Differential geometry: partial differential equations on
manifolds} (Los Angeles, CA, 1990), 1--28, Proc. Sympos. Pure
Math., 54, Part 1, Amer. Math. Soc., Providence, RI, 1993.


\bibitem{Z} Zelditch, S. {\em Real and complex zeros of Riemannian random waves}.
To appear in the Proceedings of the Conference, "Spectral Analysis in Geometry and
Number Theory on the occasion of Toshikazu Sunada's 60th birthday", to appear in the Contemp. Math. Series,
available online http://arxiv.org/abs/0803.4334

\end{thebibliography}
\end{document}